\documentclass[12pt,a4paper,twoside]{amsart}
\usepackage{a4}

\usepackage[all]{xy}
\xyoption{ps}
\xyoption{color}\UseCrayolaColors
\xyoption{dvips}
\xyoption{all}

\usepackage{mathdots, amscd, amsmath, amssymb, fixmath}
\usepackage{amscd, amsmath, amssymb}

\DeclareSymbolFont{rsfs}{OMS}{rsfs}{m}{n}
\DeclareSymbolFontAlphabet{\mathscr}{rsfs}

\renewcommand{\mathcal}{\mathscr}

\theoremstyle{remark}
\newtheorem{example}{Example}[section]
\newtheorem{rem}[example]{Remark}

\theoremstyle{definition}
\newtheorem{defn}[example]{Definition}

\theoremstyle{plain}
\newtheorem{prop}[example]{Proposition}
\newtheorem{lemma}[example]{Lemma}
\newtheorem{cor}[example]{Corollary}
\newtheorem{thm}[example]{Theorem}

\newtheorem*{claim*}{Claim}

\DeclareMathOperator{\im}{im}
\DeclareMathOperator{\id}{id}
\DeclareMathOperator{\rank}{rank}
\DeclareMathOperator{\Mat}{Mat}
\DeclareMathOperator{\GL}{GL}

\DeclareMathOperator{\SL}{SL}
\DeclareMathOperator{\Sp}{Sp}
\DeclareMathOperator{\Gr}{Gr}

\DeclareMathOperator{\Orth}{O}

\DeclareMathOperator{\End}{End}
\DeclareMathOperator{\Sym}{Sym}
\DeclareMathOperator{\diag}{diag}

\newcommand{\bbA}{\mathbb{A}} 
\newcommand{\bbZ}{\mathbb{Z}}

\newcommand{\longto}[1][]{\stackrel{#1}{\longrightarrow}}
\newcommand{\bbC}{\mathbb{C}}
\newcommand{\bbP}{\mathbb{P}}
\newcommand{\Gm}{\mathbb{G}_m}

\newcommand {\CC}{\mathbb{C}}
\newcommand{\calO}{\mathcal{O}}

\newcommand{\calL}{\mathcal{L}}
\newcommand{\calV}{\mathcal{V}}
\newcommand{\calU}{\mathcal{U}}

\newcommand{\sym}{\mathrm{sym}}
\renewcommand{\phi}{\varphi}

\newcommand{\catqot}{/\hskip-3pt/}
\newcommand{\eps}{\varepsilon}

\newcommand{\lma}{\longmapsto}
\newcommand{\lra}{\longrightarrow}
\newcommand{\ol}{\overline}
\newcommand{\q}{\quad}
\newcommand{\C}{\mathbb C}

\begin{document}
\title[Rational families of instanton bundles on $\boldsymbol{\mathbb P}^{2n+1}$]
{Rational families of instanton bundles on $\boldsymbol{\mathbb P}^{2n+1}$}

\author{L.\ Costa}
\email{costa@ub.edu}
\address{Facultat de Matem\`atiques,
Departament d'Algebra i Geometria, Gran Via de les Corts Catalanes
585, 08007 Barcelona, Spain} 

\author{N.\ Hoffmann}
\email{norbert.hoffmann@mic.ul.ie}
\address{Department of Mathematics and Computer Studies, Mary Immaculate College, South Circular Road,
Limerick, Ireland}

\author{R.M.\ Mir\'o-Roig}
\email{miro@ub.edu}
\address{Facultat de Matem\`atiques,
Departament d'Algebra i Geometria, Gran Via de les Corts Catalanes
585, 08007 Barcelona, Spain} 

\author{A.\ Schmitt}
 \email{alexander.schmitt@fu-berlin.de}
\address{Institut f\"{u}r Mathematik, Freie Universit\"{a}t Berlin, Arnimallee
3, 14195 Ber\-lin, Germany}

\thanks{The first and third author were partially supported by MTMM2010-15256, and the second and fourth 
author by SFB 647 ``Space-Time-Matter'', project C3 ``Algebraic Geometry: Deformations, Moduli and Vector Bundles''.
The final version of this paper was prepared during the last author's visit to the CRM, Bellaterra. A.\ Schmitt
would like to thank that institute for its hospitality and financial support.}



\date{\today}

\begin{abstract}
This paper is devoted to the theory of symplectic instanton bundles on an odd dimensional projective space ${\mathbb P}^{2n+1}$ with $n\ge 2$. We study the 
't Hooft instanton bundles introduced by Ottaviani and a new family of instanton bundles which generalizes one introduced on ${\mathbb P}^3$ independently 
by Rao and Skiti. The main result is the determination of the birational types of the moduli spaces of 't Hooft and of Rao--Skiti instanton bundles,
respectively. Assuming a conjecture of Ottaviani, we show that the moduli space of all symplectic instanton bundles on ${\mathbb P}^{2n+1}$ with $n\ge 2$ is 
reducible.
\end{abstract}


\setcounter{tocdepth}{1}

\maketitle


\section{Introduction}
\label{intro}
The notion of an instanton comes from mathematical physics. It denotes a solution of an equation of motion in classical field theory which describes 
a particle localized both in space and in time.  Mathematically, it is a self-dual connection on a principal bundle on a four-dimensional
Riemannian manifold. The Penrose--Ward correspondence identifies instantons on the four-dimensional sphere $S^4$ with holomorphic instanton bundles on the
three-dimensional complex projective space $\mathbb{P}^3$ (\cite{At}, \cite{AtW}). This construction turned out to be a major motivation for studying vector
bundles on complex projective spaces and similar algebraic varieties. In recent preprints, Jardim, Markushevich, Tikhomirov, and Verbitsky managed
to settle fundamental open questions on geometric properties of moduli spaces of instanton bundles on $\mathbb{P}^3$, namely smoothness, connectedness and 
rationality (\cite{JV}, \cite{MT}, \cite{T}).
\par 
The Penrose--Ward correspondence exists also on higher odd-dimensional complex projective spaces. In fact, Salamon, Corrigan, Goddard, and
Kent introduced a correspondence between self-dual connections on ${\rm Sp}_n(\C)$-bundles on the $n$-dimensional quaternionic projective space
$\mathbb{H}{\rm P}^n$ and symplectic instanton bundles on the complex projective space $\mathbb{P}^{2n+1}$, $n\ge 1$ (\cite{CS}, \cite{CGK},
\cite{Sala}). For $n\ge 2$, the knowledge about moduli spaces of symplectic instanton bundles on $\mathbb{P}^{2n+1}$ is much less complete, and it
is the aim of the present paper to make some progress in this area.
\par 
The starting point is the paper by Ottaviani \cite{ottaviani}. Ottaviani introduces the notion of symplectic 't Hooft instanton bundles on 
$\mathbb{P}^{2n+1}$, $n\ge 2$, and claims without proof that the closure of the locus of symplectic 't Hooft bundles is an irreducible component of the moduli 
space of symplectic instanton bundles. So far, we have not been able to prove this claim, but we have checked it for several values of $n$ and the instanton
number $k$ (Remark \ref{ex:Otta}). In addition, we construct an irreducible moduli space for  symplectic 't Hooft bundles. Our first main result is that this 
moduli space is stably rational or even rational for many values of $n$ and $k$ (Corollary \ref{cor:ModRat}).
\par 
Next, we define the notion of a symplectic Rao--Skiti (RS) instanton bundle on $\mathbb{P}^{2n+1}$, $n\ge 1$. Symplectic Rao--Skiti instanton bundles generalize 
bundles studied by Rao \cite{rao} and Skiti \cite{skiti} on $\mathbb{P}^3$. To our knowledge, these instanton bundles haven't been investigated for $n\ge 2$, 
so far. We supply irreducible moduli spaces for symplectic Rao--Skiti instanton bundles and prove their rationality (Corollary \ref{cor:RSInstRat}). Another
important observation is that there are symplectic Rao--Skiti instanton bundles which are not limits of 't Hooft instanton bundles (Example \ref{ex:RSNonTH}). 
Under the assumption of the claim of Ottaviani (which we have verified in some cases), this implies that the moduli space of symplectic instanton bundles is 
reducible.
\par
Next, we outline the structure of the paper. In Section 2, we fix notation and briefly recall the definition and basic properties of instanton bundles on 
projective spaces needed later on. In Section 3, we present the definition and main features of the two irreducible families of symplectic instanton bundles on 
$\bbP^{2n+1}$ that we are going to be interested in, namely, symplectic 't Hooft- and symplectic RS-instanton bundles (see Definition \ref{deftHooft} and 
\ref{defRSinstanton}, respectively). In this section we also  verify a central claim by Ottaviani concerning the deformation behavior  of 't Hooft instanton 
bundles with the help of a computer (Remark \ref{ex:Otta}). Section 4 deals with the construction (Proposition \ref{prop:ModHooft}) and the rationality 
(Corollary \ref{cor:ModRat}) of moduli spaces of 't Hooft instanton bundles. Section 5 contains the determination of the birational type of the moduli stacks 
(Corollary \ref{cor:BirRSStack}) and spaces (Corollary \ref{cor:RSInstRat}) of RS-instanton bundles.
\subsection*{Notation}
Throughout this paper, we will work over the field $\CC$ of complex numbers. Given a vector space $W$, we will denote by $\bbP(W)$ the projective space of 
lines in $W$ and set $\bbP := \bbP^{2n+1} :=\bbP (\CC^{2n+2})$. We will not distinguish between a vector bundle and its locally free sheaf of sections and use 
the definition of $\mu$-(semi)stability due to Mumford and Takemoto \cite{OSS}. Given an integer $k\ge 1$, let 
$J = \left( \begin{smallmatrix}0 & I\\-I & 0 \end{smallmatrix}\right)$ denote the standard symplectic form on $\CC^{2n+2k}$.
\section{Mathematical instanton bundles}
In this section, we recall the definition of (mathematical) instanton bundles and their description in terms of monads.
\begin{defn} 
Let $k\ge 1$ be an integer. An {\it instanton bundle with charge $k$} (for short, a {\it $k$-instanton bundle}) is a vector bundle $E$ on $\bbP$ satisfying
the following properties:
\par
i) $E$ has rank $2n$,
\par
ii) the Chern polynomial of $E$ is $c_t(E)=1/(1-t^2)^k$,
\par
iii) $E$ has natural cohomology in the range $-(2n+1)\le q \le 0$, i.e., for any $q$ in that range, there is at most one integer $i=i(q)$ such that 
$H^{i}(\bbP,E(q))\ne 0$.
\par
iv) $E$ has trivial splitting type, i.e., the restriction of $E$ to a general line is trivial.
\end{defn}
\begin{rem}
\label{rem:InstSimple}
By \cite{AO}, Proposition 2.11, any instanton bundle is simple. Nevertheless, it is an open question to determine whether it is $\mu$-stable.
\end{rem}
\begin{defn}
A vector bundle $E$ on $\bbP$ is called {\em symplectic}, if there exists an isomorphism $\phi\colon E\longrightarrow E^*$ such that $\phi ^*=-\phi$. This is 
equivalent to the existence of a non-degenerate form $\alpha \in H^0(\bbP,\bigwedge\limits^2E^*)$.
\end{defn}
For $n=1$, $E$ is a rank 2 vector bundle with $c_1(E)=0$. Thus, $\bigwedge\limits^2E= \calO_{\bbP}$ and $H^0(\bigwedge\limits^2E)$ admits a nowhere vanishing 
global section. However, when $n\ge 2$, the symplectic structure condition does not follow from the definition. In this paper, we will restrict our attention 
mainly to symplectic instanton bundles. Using the Beilinson spectral sequence \cite{Be}, we get the following well-known and useful correspondence  between 
symplectic instanton bundles and self-dual monads:
\begin{prop}
\label{prop:Monadology}
{\rm i)} Any symplectic $k$-instanton bundle $E$ over $\bbP$ is the cohomology of a symplectic monad
\begin{equation}
\label{eq:monad}
\begin{CD}
  \calO_{\bbP}( -1)^{\oplus k} @> JA^t >> \calO_{\bbP}^{2n+2k} @> A >> \calO_{\bbP}(
1)^{\oplus k}
\end{CD}
\end{equation}
where the matrix
\[
  A \in \Mat_{k \times (2n+2k)} \Big( H^0\bigl(\bbP, \calO_{\bbP}(1)\bigr) \Big)
\]
of linear forms has full rank $k$ at every point of $\bbP$ and satisfies $AJA^t = 0$.

Conversely, any such matrix $A$ yields a symplectic monad \eqref{eq:monad},
whose cohomology $E$ is a symplectic $k$-instanton bundle if $E$ has trivial splitting type.
\par
{\rm ii)} Abbreviate a monad with cohomology bundle $E$ as in \eqref{eq:monad} by $M^\bullet_E$. Then, for two such monads $M^\bullet_E$ and $N^\bullet_F$, one has
\[
 {\rm Hom}_{\calO_{\mathbb{P}}}(E,F)={\rm Hom}_{{\rm Kom}(\mathbb{P})}(M^\bullet_E, N^\bullet_F)
\]
where ${\rm Hom}_{{\rm Kom}(\mathbb{P})}$ means morphisms in the category of complexes ${\rm Kom}(\mathbb{P})$.
\end{prop}
\begin{proof}
i) See \cite{OS}, Corollary 1.4 and Lemma 1.5, and \cite{Hof}, Proposition 2.4.
\par
ii) See \cite{OS2}, Proposition 1.3, or \cite{Hof}, Proposition 2.4.
\end{proof}
\begin{prop}
\label{boundsections}
Let $E$ be a $k$-instanton bundle on $\bbP$ and let $L\subset \bbP$ be a linear subspace of dimension $r$. Then, $ h^0(E_{|L}) \leq 2n+k-r$.
\end{prop}
\begin{proof}
Assume that there is a linear subspace $L\subset \bbP$ of dimension $r$, such that $h^0(E_{|L}) \geq 2n+k-r+1$ and consider the monad
\[
\begin{CD}
\calO_{\bbP}( -1)^{\oplus k} @> JA^t>> \calO_{\bbP}^{2n+2k} @>A>> \calO_{\bbP}(1)^{\oplus k}
\end{CD}
\] 
associated with $E$. After changing basis, if necessary, we can assume that $\bbP$ has homogeneous coordinates $x_0, x_1,..., x_{2n+1}$, $L$ is given by 
$x_{r+1}=\cdots =x_{2n+1}=0$ and
\[
A _{|L}=(A_1  |  A_2) \in   \Mat_{k \times (2n+2k)} \Big(H^0\bigl(\bbP, \calO_{L}(1)\bigr)\Big)
\]
with $A_2=(0)  \in   \Mat_{k \times (2n+k-r+1)}( H^0(\bbP, \calO_{L}( 1)))$. By \cite{BV}, Theorem 2.1, the homogeneous ideal $I\subset \CC[x_0, x_1, ... ,x_r]$ 
that is generated by the maximal minors of $A_1 \in   \Mat_{k \times (k+r-1)}( H^0(\bbP, \calO_{L}( 1)))$ has height $\le r$. This implies that there exist closed points 
$x\in L$, such that ${\rm rank}(A(x))<k$,  contradicting  the fact that $A$ is the matrix which defines the symplectic monad associated with $E$.
\end{proof}
\begin{defn}
Given a symplectic instanton bundle $E$ on $\bbP$ and a linear subspace $L\subset \bbP$ of dimension $r$, we will say that $L$ is {\em unstable with maximal order of 
instability}, if $h^0(E_{|L}) = 2n+k-r$.
\end{defn}
\begin{rem}
The upper bound given in the above proposition is sharp. Indeed, we will see that it is attained by symplectic RS-instanton bundles (see Definition \ref{defRSinstanton}). 
More precisely, for any symplectic RS-instanton bundle $E$ on $\bbP$, we will prove the existence of an $n$-dimensional linear subspace $L\subset \bbP$ with maximal order 
of instability (Proposition \ref{boundsections-sharp}).
\end{rem}
\section{Symplectic 't Hooft- and RS-instanton bundles}
This section is devoted to the introduction of the two kinds of instanton bundles that will be the subject of our research, namely, symplectic 't Hooft- and symplectic 
RS-instanton bundles on $\bbP$.
\par
Symplectic 't Hooft bundles on $\bbP^3$ were introduced by A.\ Hirschowitz and M.S.\ Narasimhan in \cite{HN}. They constructed them as the bundles associated via 
Serre's correspondence with disjoint unions of lines and, as a main achievement, proved their unobstructedness. The notion of symplectic 't Hooft bundles on higher 
odd dimensional projective spaces  $\bbP$ is due to Ottaviani \cite{ottaviani}, and we start this section recalling it. To this end, we fix integers $n, k \geq 1$. Let
\begin{equation}
\label{eq:a}
a \in \Mat_{k \times (n+k)}( \bbC)
\end{equation}
be a matrix of scalars, and let
\begin{align}
  \label{eq:D}  D  & := \diag( l_1,  \ldots, l_{n+k}),  \quad l_1,  \ldots, l_{n+k}
 \in H^0\bigl(\bbP,\calO_{\bbP}(1)\bigr),\\
  \label{eq:D'} D' & := \diag( l_1', \ldots, l_{n+k}'), \quad l_1', \ldots,
l_{n+k}' \in H^0\bigl(\bbP,\calO_{\bbP}(1)\bigl),
\end{align}
be diagonal matrices with entries in $H^0(\bbP,\calO_{\bbP}(1))$. We consider the matrix
\begin{equation}
\label{eq:A2}
  A := a \cdot (D|D') \in \Mat_{k \times (2n+2k)}\Big(H^0\bigl(\bbP,\calO_{\bbP}(1)\bigr)\Big)
\end{equation}
as a morphism of vector bundles from $\calO_{\bbP}^{2n+2k}$ to $\calO_{\bbP}( 1)^{\oplus k}$.
\begin{prop}[(Ottaviani \cite{ottaviani}, Section 3)]
\label{prop:matrices}
For every choice of $a$,  $l_j, l_j'$, we have:
\par
{\rm i)} $A J A^t = 0$.
\par
{\rm ii)} The sheaf $\ker A( 1) \subseteq \calO_{\bbP}( 1)^{\oplus (2n+2k)}$ has at least
$n+k$ global sections.
\par
\noindent Moreover, if $a$ and the $l_j, l_j'$ are general,\footnote{i.e., the tuple $(a, l_j, l_j', j=1,...,n+k)$ lies in a dense open subset of the affine space 
$\Mat_{k \times (n+k)}( \bbC)\oplus H^0(\bbP,\calO_{\bbP}(1))^{\oplus 2(n+k)}$} then we have:
\par
{\rm iii)} The matrix $A$ has rank $k$ at every point of $\bbP$.
\par
{\rm iv)} If $k \geq 3$, then $\ker A( 1) \subseteq \calO_{\bbP}( 1)^{\oplus (2n+2k)}$
has exactly $n+k$ global sections.
\end{prop}
\begin{proof}
i) This follows from the calculation
  \[
    (D|D')\cdot J\cdot (D|D')^t = (D|D')\cdot \left( \dfrac{D'}{-D} \right) = D D' - D' D = 0.
  \]
\par 
ii) The previous calculation also implies $A\cdot J\cdot  (D|D')^t = 0$. Consequently, the $n+k$ columns of the matrix $J\cdot (D|D')^t$ are elements in $\ker A( 1)$.
We may assume that the linear forms $l_j$ are all nonzero. Then, these columns are linearly independent.
\par 
For the rest of the proof, let all $(k \times k)$-minors of $a$ be nonzero, and choose a decomposition
  \[
    H^0\bigl(\bbP, \calO_{\bbP}(1)\bigr) = V \oplus W \quad\text{with } \dim V = \dim W =
n+1.
  \]
\par 
iii) Choose $l_1, \ldots, l_{n+k} \in V$, such that any $n+1$ of them are a basis of $V$. Consider a point in $\bbP$ where at least one $v \in V$ does not vanish. 
Then, at most $n$ of the forms $l_1, \ldots, l_{n+k}$ vanish there, so we can find $k$ of them that do not. The corresponding $(k \times k)$-minor of 
$a \cdot \diag( l_1, \ldots, l_{n+k})$ is nonzero at this point.
\par 
Choosing $l_1', \ldots, l_{n+k}' \in W$ similarly, we can achieve that $a \cdot \diag( l_1', \ldots, l_{n+k}')$ has rank $k$ at all points where at least one 
$w \in W$ does not vanish. This covers all points in $\bbP$, thereby proving that iii) holds for some $a$ and $l_j, l_j'$.
\par 
iv) We will again take $l_1, \ldots, l_{n+k} \in V$ and $l_1', \ldots, l_{n+k}'\in W$. A global section
  \[
    \left( \dfrac{b \oplus c}{b' \oplus c'} \right) \in H^0\bigl(\bbP, \calO_{\bbP}^{2n+2k}(1)\bigr), \quad b, b' \in V^{n+k}, \quad c, c' \in W^{n+k},
  \]
is then in $\ker A( 1)$ if and only if it satisfies the system of linear equations
\begin{align}
    \label{eq:1} aDb          & = 0 \in ( \Sym^2 V)^k,\\
    \label{eq:2} aD'c'        & = 0 \in ( \Sym^2 W)^k,\\
    \label{eq:3} a(Dc + D'b') & = 0 \in ( V \otimes W)^k.
\end{align}
Choosing bases $v_1, \ldots, v_{n+1} \in V$ and $w_1, \ldots, w_{n+1} \in W$, we write
\begin{align*}
    D  & = D_1  v_1 + \cdots + D_{n+1}  v_{n+1}, & \text{with}\q & D_1,
\ldots, D_{n+1}  \in \Mat_{(n+k) \times (n+k)}( \bbC),\\
    D' & = D_1' w_1 + \cdots + D_{n+1}' w_{n+1}, & \text{with}\q & D_1',
\ldots, D_{n+1}' \in \Mat_{(n+k) \times (n+k)}( \bbC).
\end{align*}
Specifying the linear forms $l_1, \ldots, l_{n+k} \in V$ and $l_1', \ldots, l_{n+k}' \in W$ is equivalent to specifying the diagonal matrices $D_1, \ldots, D_{n+1}$ 
and $D_1', \ldots, D_{n+1}'$. We also write
\begin{align*}
    b  & = b_1  v_1 + \cdots + b_{n+1}  v_{n+1}, & \text{with}\q & b_i  =
  (b_{i, 1},  \ldots, b_{i, n+k})^t  \in \bbC^{n+k},\\
    c  & = c_1  w_1 + \cdots + c_{n+1}  w_{n+1}, & \text{with}\q & c_i  =
  (c_{i, 1},  \ldots, c_{i, n+k})^t  \in \bbC^{n+k},
\end{align*}
and similarly for $b'$ and $c'$.
\par 
Suppose $k \geq 2$. We claim that Equation \eqref{eq:1} has only the trivial solution $b = 0$, if $l_1, \ldots, l_{n+k} \in V$ are general. It suffices to check this in 
one special case, say
\[
    l_j := \begin{cases} v_j, & \text{for } j \leq n+1,\\ v_1 + \cdots + v_{n+1}, &
\text{for } j \geq n+2 \end{cases}.
\]
This choice translates into
\[
    D_i = \diag( \underbrace{0, \ldots, 0}_{i-1}, 1, \underbrace{0, \ldots,
0}_{n+1-i}, \underbrace{1, \ldots, 1}_{k-1}).
\]
The coefficient for $v_i^2$ in Equation \eqref{eq:1} reads $a D_i b_i = 0$. As the relevant $(k \times k)$-minor of $a$ is nonzero, this implies $D_i b_i = 0$
and, hence,
\begin{equation} 
\label{eq:b_ii}
    b_{i, i} = b_{i, n+2} = b_{i, n+3} = \cdots = b_{i, n+k} = 0.
\end{equation}
Given indices $1 \leq i_1 < i_2 \leq n+1$, the coefficient for $v_{i_1} v_{i_2}$ in Equation \eqref{eq:1} reads
\[
    a (D_{i_1} b_{i_2} + D_{i_2} b_{i_1}) = 0 \in \bbC^k.
\]
Using \eqref{eq:b_ii}, we see that $D_{i_1} b_{i_2} + D_{i_2} b_{i_1} \in \bbC^{n+k}$ has at most $2 \leq k$ nonzero components; as $(k \times k)$-minors of $a$ are 
nonzero, we can conclude
\[
    D_{i_1} b_{i_2} + D_{i_2} b_{i_1} = 0 \in \bbC^{n+k}
\]
and, hence, $b_{i_1, i_2} = b_{i_2, i_1} = 0$. This shows that $b = 0$ is the only solution of \eqref{eq:1} for this particular choice, and, consequently, 
for a general choice, of $l_1, \ldots, l_{n+k} \in V$. Similarly, $c' = 0$ is the only solution of \eqref{eq:2} for a general choice of $l_1', \ldots, l_{n+k}' \in W$.
\par 
Suppose $k \geq 3$. It remains to show that Equation \eqref{eq:3} has only $n+k$ linearly independent solutions $(b', c)$, if $l_1, \ldots, l_{n+k} \in V$ and 
$l_1', \ldots, l_{n+k}' \in W$ are general. Again, it suffices to check this in one special case, say
\[
    l_j  := \begin{cases} v_j, & \text{for } j \leq n
            \\
              v_n, & \text{for } j = n+1
            \\
             v_{n+1}, & \text{for } j \geq n+2
        \end{cases}
\quad\text{and}\quad
    l_j' :=
\begin{cases}
  w_j, & \text{for } j \leq n+1
\\
  w_n, & \text{for } j = n+2
\\
  w_{n+1}, & \text{for } j \geq n+3
\end{cases}.
\]
This choice translates into
\begin{align*}
    D_i = D_i' & = \diag( \overbrace{0, \ldots, 0}^{i-1}, 1, \overbrace{0, \ldots,
    0, 0, 0, 0, 0, \dots, 0}^{n+k-i}),\quad \text{ for } i \leq n-1,\\
    D_n  & = \diag( 0, \ldots, 0, 0, 0, \ldots, 0, 1, 1, 0, 0, \dots, 0),\\
    D_{n+1}  & = \diag( 0, \ldots, 0, 0, 0, \ldots, 0, 0, 0, 1, 1, \dots, 1),\\
    D_n' & = \diag( 0, \ldots, 0, 0, 0, \ldots, 0, 1, 0, 1, 0, \dots, 0),\\
    D_{n+1}' & = \diag( \underbrace{0, \ldots, 0, 0, 0, \ldots, 0}_{n-1}, 0, 1, 0,
\underbrace{1, \dots, 1}_{k-2});
\end{align*}
note that $D_1 + \cdots + D_{n+1}$ and $D_1' + \cdots + D_{n+1}'$ both equal the identity matrix.
\par 
Given indices $1 \leq i_1, i_2 \leq n+1$, the coefficient for $v_{i_1} \otimes w_{i_2}$ in Equation \eqref{eq:3} reads
\[
    a (D_{i_1} c_{i_2} + D_{i_2}' b_{i_1}') = 0 \in \bbC^k.
\]
Due to the particular choice we have made, $D_{i_1} c_{i_2} + D_{i_2}' b_{i_1}' \in \bbC^{n+k}$ has at most $k$ nonzero components; as $(k \times k)$-minors of $a$ are 
nonzero, we can conclude
\[
    D_{i_1} c_{i_2} + D_{i_2}' b_{i_1}' = 0 \in \bbC^{n+k}.
\]
Taking separately the sum over all $i_2$ and the sum over all $i_1$, we get
\[
    b_{i_1}' = - D_{i_1}( c_1 + \cdots + c_{n+1}) \quad\text{and}\quad c_{i_2} = -
D_{i_2}'( b_1' + \cdots + b_{n+1}').
\]
Plugging the sum over all $i_1$ of the former equation into the latter, we conclude
\[
    c_{i_2} = D_{i_2}'( c_1 + \cdots + c_{n+1}).
\]
This shows that every solution $(b', c)$ of Equation \eqref{eq:3} satisfies
\[
    b' = -D( c_1 + \cdots + c_{n+1}) \quad\text{and}\quad c = D'( c_1 + \cdots +
c_{n+1}).
\]
As $c_1 + \cdots + c_{n+1} \in \bbC^{n+k}$, there are only $n+k$ linearly independent solutions.
\end{proof}
\begin{rem} 
\label{k=1,2} 
The hypothesis $k\ge 3$ in Proposition 3.1, iv), cannot be dropped. In fact, it is trivial to check that, if $k=1$, then 
$\ker A( 1) \subseteq \calO_{\bbP}( 1)^{\oplus (2n+2k)}$ has exactly $2n^2+3n+1 $ global sections; and it follows from \cite{AO}, Theorem 3.14, that, if $k=2$, then 
$h^0(\bbP,\ker A(1))= 2n+2$.
\end{rem}
\begin{defn}[\cite{ottaviani}, Section 3] \label{deftHooft}  
A symplectic $k$-instanton bundle $E$ on $\bbP$ given by a monad
\[
\begin{CD}
\calO_{\bbP}( -1)^{\oplus k} @> JA^t>> \calO_{\bbP}^{2n+2k} @>A>> \calO_{\bbP}( 1)^{\oplus k}
\end{CD}
\]
is called an \emph{'t Hooft instanton bundle}, if $A$ is of the form \eqref{eq:A2} introduced above.
\end{defn}
\begin{rem}
i) Proposition \ref{prop:matrices}, i) and iii), shows that general data
\[
  a \in \Mat_{k \times (n+k)}( \bbC) \quad\text{and}\quad l_1,  \ldots, l_{n+k},
l_1', \ldots, l_{n+k}' \in H^0\bigl(\bbP, \calO_{\bbP}(1)\bigr)
\]
actually define a symplectic 't Hooft $k$-instanton bundle $E$, via \eqref{eq:D}, \eqref{eq:D'}, \eqref{eq:A2}
and \eqref{eq:monad}. The condition on the trivial splitting type follows from the other conditions
(\cite{ottaviani}, Corollary 3.5).
\par 
ii) It follows from \cite{ottaviani}, Theorem 3.6, that any 't Hooft instanton bundle on $\bbP$ is stable.
\end{rem}
The following statement is also contained in \cite{ottaviani}, Proposition 3.3 and Theorem 3.7.
\begin{prop}
\label{sectionsthooft}
Let $E$ be a symplectic 't Hooft $k$-instanton bundle. Then
\begin{equation}
  \label{eq:sections}
    h^0\bigl(\bbP, E(1)\bigr) \geq n
\end{equation}
with equality if $k \geq 3$ and $E$ is general.
\end{prop}
\begin{proof}
Let $E$ be a symplectic $k$-instanton bundle, given as the cohomology of the monad \eqref{eq:monad}. Then, $E( 1)$ is the cohomology of a monad
\[
  \begin{CD}
    \calO_{\bbP}^{\oplus k} @>>> \calO_{\bbP}( 1)^{\oplus (2n+2k)} @> A( 1) >> \calO_{\bbP}( 2)^{\oplus k}.
  \end{CD}
\]
Using $H^1(\bbP, \calO_{\bbP}^{\oplus k}) = 0$, this implies, in particular,
\[
    h^0\bigl(\bbP, E(1)\bigr) = h^0\bigl(\bbP, \ker A( 1)\bigr) - k.
\]
Hence, the claim follows from Proposition \ref{prop:matrices}, ii) and iv).
\end{proof}
\begin{rem}
The hypothesis $k\ge 3$ in the above proposition cannot be dropped. In fact, let $E$ be a symplectic $k$-instanton bundle on $\bbP^{2n+1}$. By Remark \ref{k=1,2}, 
we know that if $k=1$ ($k=2$) then $h^0(\bbP, \ker A(1))=2n^2+3n+1$ ($h^0(\bbP, \ker A(1))=2n+2$) and, hence, $h^0(\bbP, E(1))=2n^2+3n$ ($h^0(\bbP, E(1))=2n$, 
respectively).
\end{rem}
\begin{rem}[(Deformations of 't Hooft instantons)]
\label{ex:Otta}
Let $E$ be a general symplectic 't Hooft $k$-instanton bundle over $\bbP$. Ottaviani claims in \cite{ottaviani} that every infinitesimal deformation of $E$ as 
a symplectic bundle comes from a deformation as a symplectic 't Hooft bundle for $n \geq 2$ and $k \geq 9$. This would mean that (the closure of) the locus of 
't Hooft bundles is an irreducible component of the moduli space of symplectic instanton bundles and be very interesting in view of our rationality result 
\ref{cor:ModRat}. Unfortunately, there is not much of a proof for that in \cite{ottaviani}, just a reference to a future paper which has not yet appeared.
\par
The claim of Ottaviani in question is equivalent to
\begin{eqnarray*}
 && \displaystyle{\dim_{\bbC} \left\{\, X \in \Mat_{k \times (2n+2k)}\big(H^0\bigl(\bbP, \calO_{\bbP}(1)\bigr)
\big)\,|\, A\cdot J\cdot X^t \text{ is symmetric } \,\right\}}
\\
 && \displaystyle{= (5kn + 4n^2) + \dim\bigl( \GL_k(\C) \times \Sp_{2n+2k}(\C)\bigr) }
\\
&& = (n+k)\cdot (6n+3k+1)
\end{eqnarray*}
where $A$ is still a matrix of the form \eqref{eq:A2} above, with general parameters $a$ and $l_j, l_j'$. Using \sc Maple\rm, we checked this claim for the following 
values of $(n,k)$:
\[
 (2,9),\ (2,10),\ (3,6),\ (3,7),\ (3,8),\ (3,9),\ (4,5),\ (4,6),\ (4,7),\ (5,5),\ (5,6).
\]
In these cases, the closure of the locus of 't Hooft bundles is an irreducible component of the moduli space of symplectic instanton bundles. (Note that the reducibility 
of the moduli space of all instanton bundles was established in \cite{AO00}.)
\end{rem}
We devote the last part of this section to the introduction of symplectic RS-instanton bundles on $\bbP$. For $ n=1 $, symplectic RS-instanton bundles were introduced and 
studied independently by Rao in \cite{rao} and by Skiti in \cite{skiti}. They are characterized as rank 2 $k$-instanton bundles on $\bbP^3$ with  a jumping line of 
maximal order, namely, of order $k$. The moduli space of $k$-instantons on $\bbP^3$ has a natural stratification by the maximal order of a jumping line for $E$, and 
Rao and Skiti proved that the stratum  that consists of $k$-instantons having a jumping line of the maximal possible order $k$ is an irreducible rational variety of 
dimension $6k+2$. Our next goal will be to introduce the notion of symplectic RS-instanton bundles on $\bbP$ and study some of their properties.
\par
Again, we fix integers $k,n \ge 1$. Let
\begin{equation}
\label{eq:aa}
  H:=(h_{ij}) \in \Mat_{k \times (n+k)}\Bigl(H^0\bigl(\bbP, \calO_{\bbP}(1)\bigr)\Bigr)
\end{equation}
be a persymmetric matrix of linear forms $h_{ij} \in H^0(\bbP,\calO_{\bbP}(1))$, i.e., a matrix such that $h_{ij}=h_{st}$ if  $i+j=s+t$. Let
\begin{equation}
\label{eq:L}
L   := \{\, f_0=f_1=\cdots=f_n=0\, \}
\end{equation}
be a linear $n$-space which contains no zeros of the maximal minors of $H$.  We consider the matrices
\[
F := \left (
\begin{array}{cccccccccc}
f_{0} & f_{1} & \cdots &f_{n}  & 0 & \cdots& \cdots & 0
\\ 0 &
f_{0} & f_{1} & \cdots & f_{n}  & 0  &  \cdots &0
\\
0 & 0 & f_{0} & f_{1} & \cdots & f_{n} & \ddots & \vdots
\\
\vdots & \cdots & \ddots & \ddots & \ddots   & \cdots & \ddots &  0
\\
0  & \cdots & \cdots & 0 & f_0  & f_1 & \cdots & f_{n}
\end{array}
\right )
\]
and
\begin{equation}
\label{eq:RSA}
A:=(F|H) \in \Mat_{k \times (2n+2k)}\Big(H^0\big(\bbP, \calO_{\bbP}(1)\big)\Big).
\end{equation}
The matrix $A$ will be considered  as a morphism of vector bundles from $\calO_{\bbP}^{2n+2k}$ to $\calO_{\bbP}( 1)^{\oplus k}$ and it satisfies the following 
properties.
\begin{prop}
\label{propertiesSR}
For general choices of $f_s \in H^0(\bbP,\calO_{\bbP}(1))$, $0 \le s \le n$, and of the persymmetric matrix $H$
in \eqref{eq:aa}, we have:
\par
{\rm i)} $A J A^t = 0$.
\par
{\rm ii)} The sheaf $\ker A _L\subseteq \calO_{L}^{2n+2k}$ has exactly $n+k$ global sections.
\par
{\rm iii)} The matrix $A$ has rank $k$ at every point $x$ of $\bbP$.
\end{prop}
\begin{proof}
i) This part follows from the calculation
\begin{equation*}
    (F|H)\cdot J\cdot (F|H)^t = (F|H)\cdot \left( \dfrac{H^t}{-F^t} \right) = F H^t - HF^t = 0.
\end{equation*}
\par 
ii) This is due to the fact that
\[
A_L =(0|H_L)
\]
has exactly $n+k$ independent syzygies of degree 0. (In general, we denote by ${\rm Syz}_d(A)$ the space of relations $\sum _{i=1}^q \alpha_i\cdot c_i=0$ among 
the column vectors  $c_1,...,c_q$ of the matrix $A$ with ${\rm deg}(\alpha_i)=d$, $i=1,...,q$, $d\ge 0$.)
\par 
iii) We deduce this from the fact that the linear space $L = \{\, f_0=f_1=\cdots=f_n=0\, \} $ contains no zeros of the maximal minors of $H$.
\end{proof}
\begin{defn}
\label{defRSinstanton}
A symplectic $k$-instanton bundle $E$ on $\bbP$ given by a monad
\[
\begin{CD}
\calO_{\bbP}( -1)^{\oplus k} @>JA^t>> \calO_{\bbP}^{2n+2k} @>A>> \calO_{\bbP}( 1)^{\oplus k}
\end{CD}
\]
is called an \emph{RS-instanton bundle}, if $A$ is of the form \eqref{eq:RSA} introduced above.
\end{defn}
\begin{rem}
\label{rem:RSStable}
i) Examples of symplectic RS-instanton bundle $E$ on $\bbP$ include the special symplectic instanton bundles
introduced by Spindler and Trautmann in \cite{ST}, Definition 4.1.
\par 
ii) Since any symplectic RS-instanton on $\bbP$ is determined by the choice of general linear forms
$f_s \in H^0(\bbP,\calO_{\bbP}(1))$, $0 \le s \le n$, and of a general persymmetric matrix $H$ as in
\eqref{eq:aa}, we immediately get that the family of symplectic RS-instanton bundles is irreducible.
This irreducibility together with the fact that special symplectic instanton bundles are stable \cite{AO}
allows us to conclude that there exists a non-empty open subset of stable symplectic RS-instanton bundles.
\end{rem}
In the next proposition we are going to prove that the bound given in Proposition \ref{boundsections} is sharp. Indeed, we have
\begin{prop}
\label{boundsections-sharp}
Let $E$ be a symplectic RS-instanton bundle on $\bbP$. There exists an unstable $n$-dimensional linear subspace $L\subset \bbP$ with maximal order of instability. 
Moreover, if $E$ is general, then $L$ is unique.
\end{prop}
\begin{proof}
Let $E$ be  a symplectic RS-instanton bundle $E$ on $\bbP$ associated with an $n$-space $L= \{\,f_0=f_1=\cdots=f_n=0\, \}$. By Proposition \ref{boundsections}, 
$h^0(E|_L) \leq n+k$. On the other hand, it is easy to see, using Proposition \ref{propertiesSR} and the display of the monad associated with $E$ 
(see \cite{OSS}, page 239), that $h^0(E|_L)$ equals the number of linearly independent linear syzygies of the matrix $A$ evaluated on $f_0=f_1=\cdots=f_n=0$. 
Hence, $h^0(E|_L) \geq n+k$ and thus we get the equality.
\par 
It remains to prove the uniqueness. Let $E$ be a general symplectic RS-instanton bundle on $\bbP$. After a change of coordinates if necessary we can assume without 
loss of generality that $E$ is the cohomology of the monad
\begin{equation}
\label{monadL}
\begin{CD}
\calO_{\bbP}( -1)^{\oplus k} @> JA^t >> \calO_{\bbP}^{2n+2k} @>A>> \calO_{\bbP}( 1)^{\oplus k}
\end{CD}
\end{equation}
where $A=(F|H)$ with $H=(h_{ij})$ a persymmetric matrix of linear forms and
\[
  F = \left (
\begin{array}{cccccccccc}
x_{0} & x_{1} & \cdots &x_{n}  & 0 & \cdots& \cdots & 0
\\ 0 &
x_{0} & x_{1} & \cdots & x_{n}  & 0  &  \cdots &0
\\
0 & 0 & x_{0} & x_{1} & \cdots & x_{n} & \ddots & \vdots
\\
\vdots & \cdots & \ddots & \ddots & \ddots   & \cdots & \ddots &  0
\\
0  & \cdots & \cdots & 0 & x_0  & x_1 & \cdots & x_{n}
\end{array}
\right),
\]
$x_0,..., x_{2n+1}$ being homogeneous coordinates on $\bbP$. Let $L$ be the $n$-space $\{\, x_0=x_1=\cdots=x_n=0\, \}$. It follows from the first part of the proof 
that $h^0(E|_{L})=2n+k-n=n+k$. The fact that $L$ is the unique $n$-space with this property results from the following
\begin{claim*}
For any $n$-space $L' \subset \bbP$, $L'\neq L$, $h^0(E|_{L'})<n+k$.
\end{claim*}
We now prove the claim. Since $L' \neq L$, there exists $i \in \{\,0, ..., n \,\}$, such that $L' \nsubseteq \{\, x_i=0\, \}$. Without loss of generality assume 
that $i=0$. Associated with the monad (\ref{monadL}), we have the two short exact sequences
\[
\begin{CD}
0 @>>> K @>>> \calO_{\bbP}^{2n+2k} @>A>> \calO_{\bbP}( 1)^{\oplus k} @>>> 0
\end{CD}
\]
and
\[
\begin{CD}
0 @>>> \calO_{\bbP}( -1)^{\oplus k}  @>>> K @>>> E @>>> 0
\end{CD}
\]
where $K$ stands for the kernel of $A$. Restricting both sequences to $L'$ and taking cohomology, we see that $h^0(E|_{L'})=h^0(K|_{L'})$. Since 
$L' \nsubseteq \{\, x_0=0\, \}$ and the linear forms $h_{ij}$ are general, we have $h^0(K|_{L'})\leq n+k-1$. This finishes the proof of the claim and the proposition.
\end{proof}
By means of the following example, we will see that there are symplectic RS-instanton bundles which are not limits of symplectic 't Hooft instanton bundles.
\begin{example}
\label{ex:RSNonTH}
Fix homogeneous coordinates $x_0, ...,x_n$, $y_0, ...,y_n$ on $\bbP$. Given a complex number $\eps$, we consider the persymmetric matrix
\[
H_{\eps}:= \left(
\begin{array}{rccccccl}
 \eps y_1 & 0 & \cdots & 0 & y_0 & \cdots& y_{n-1} & y_n
\\
0 & \cdots & 0 & y_0 & \cdots&  y_{n-1} & y_n & 0
\\
\vdots & \iddots & \iddots & \cdots&  \iddots & \iddots & \iddots & \vdots
\\
0 &  y_0 & \cdots&  y_{n-1} & y_n & 0 &\cdots &0
\\
 y_0 & \cdots & y_{n-1} & y_n & 0 & \cdots & 0 & \eps y_1
\end{array}
\right)
\]
and the $((2n+2k)\times k)$-matrix $A_{\eps}=(F|H_{\eps})$ with
\[
F =\left (
\begin{array}{cccccccccc}
x_{0} & x_{1} & \cdots &x_{n}  & 0 & \cdots& \cdots & 0
\\ 0 &
x_{0} & x_{1} & \cdots & x_{n}  & 0  &  \cdots &0
\\
0 & 0 & x_{0} & x_{1} & \cdots & x_{n} & \ddots & \vdots
\\
\vdots & \cdots & \ddots & \ddots & \ddots   & \cdots & \ddots &  0
\\
0  & \cdots & \cdots & 0 & x_0  & x_1 & \cdots & x_{n}
\end{array}
\right ).
\]
The matrix $A_0$ has rank $k$ at every point of $\bbP$ and defines a symplectic $k$-instanton bundle,
namely a special symplectic instanton bundle in the sense of \cite{ST}. Therefore, if $\eps$ is small enough,
then $A_{\eps}$ still defines a symplectic $k$-instanton bundle. Assume $n \geq 2$ and 
$k \geq 3$. If $\eps \neq 0$ is small enough, then the RS-instanton bundle $E_{\eps}$  defined by $A_{\eps}$ is not a limit of 't Hooft instanton bundles. 
Indeed, this will follow from Proposition \ref{sectionsthooft} and the following 
\begin{prop}
  Assume $n \geq 2$ and $k \geq 3$. If $\eps \neq 0$ is sufficiently small, then
  the RS-instanton bundle $E_{\eps}$ defined by the matrix $A_{\eps}$ satisfies
  $H^0(\bbP, E_{\eps}(1))=0$.
\end{prop}
\begin{proof}
By definition, $E_{\eps}$ is the cohomology of the monad
\[
 \begin{CD}
    \calO_{\bbP}( -1)^{\oplus k} @>JA_{\eps}^t>> \calO_{\bbP}^{2n+2k} @>A_{\eps}>>
\calO_{\bbP}( 1)^{\oplus k}.
 \end{CD}
\]
We have the two short exact sequences
\[
 \begin{CD}
0 @>>>  K_{\eps} @>>> \calO_{\bbP}^{2n+2k} @>A_{\eps}>> \calO_{\bbP}( 1)^{\oplus k},
 \end{CD}
\]
\[
 \begin{CD}
0 @>>> \calO_{\bbP}( -1)^{\oplus k} @>>> K_{\eps} @>>> E_{\eps} @>>> 0,
 \end{CD}
\]
where $K_{\eps}$ denotes the kernel of $A_{\eps}$. Twisting these short exact sequences by $\calO_{\bbP}(1)$ and taking cohomology, we find that 
$H^0(\bbP, K_{\eps}(1))$ coincides with the space
\[
 {\rm Syz}_1(A_{\eps})={\rm
ker}\Bigl(H^0\bigl(\mathbb{P},\calO_{\mathbb{P}}(1)^{\oplus (2n+2k)}\bigr)
  \stackrel{A_{\eps}}{\lra} H^0\bigl(\mathbb{P},\calO_{\mathbb{P}}(2)^{\oplus
k}\bigr)\Bigr)
\]
of all syzygies of $A_{\eps}$ of degree one, and that its dimension ${\rm syz}_1(A_{\eps})$ satisfies
\[
h^0\big(\bbP, E_{\eps}(1)\big)= {\rm syz}_1(A_{\eps})-k.
\]
Let us prove that the following independent vectors span the space ${\rm Syz}_1(A_{\eps})$:
\begin{align*}
v_1 & =(-\eps y_1,\overbrace{0, ...,0}^{k-2},-y_0, ..., -y_n,
  x_0, ..., x_n, \overbrace{0, ...,0}^{k-1})^{t},\\
v_i & =(\overbrace{0, ...,0}^{k-i},-y_0, ..., -y_n, \overbrace{0, ...,0}^{2i-2},
  x_0, ..., x_n, \overbrace{0, ...,0}^{k-i})^{t}, \quad 1 < i < k,\\
v_k & =(-y_0, ..., -y_n,\overbrace{0, ...,0}^{k-2}, -\eps y_1,
  \overbrace{0, ...,0}^{k-1}, x_0, ..., x_n )^{t}.
\end{align*}
Since $A_{\eps} \cdot v_i=0$, it is clear that $v_i$ belongs to ${\rm Syz}_1(A_{\eps})$, for $i = 1, \ldots, k$. Let us see that the other inclusion also 
holds. We consider an arbitrary vector
\[
v=(\alpha_1, ..., \alpha_{n+k}, \beta_1, ..., \beta_{n+k})^t \in {\rm Syz}_1(A_{\eps})
\]
with $\alpha_q, \beta_q \in H^0\big(\bbP, \calO_{\bbP}(1)\big)$ for $q = 1, \ldots, n+k$. Since $A_{\eps} \cdot v = 0$, we have
\begin{equation} 
\label{eq:Av=0}
\begin{array}{@{}c@{\, + \,}c@{\, + \cdots + \,}c@{\, + \,}c@{\, + \cdots + \,}l@{}}
  \alpha_1 x_0 & \alpha_2 x_1 & \alpha_{n+1} x_n & \beta_k y_0 & \beta_{n+k} y_n +
\eps \beta_1 y_1 = 0,\\[1ex]
  \alpha_i x_0 & \alpha_{1+i} x_1 & \alpha_{n+i} x_n & \beta_{k+1-i} y_0 &
\beta_{n+k+1-i} y_n = 0,
    \hspace{1ex} 1 < i < k,\\[1ex]
  \alpha_k x_0 & \alpha_{1+k} x_1 & \alpha_{n+k} x_n & \beta_1 y_0 & \beta_{n+1} y_n
+ \eps \beta_{n+k} y_1 = 0.
\end{array}
\end{equation}
Using the basis $x_0, \ldots, x_n, y_0, \ldots, y_n$ of $H^0\big(\bbP, \calO_{\bbP}(1)\big)$, we write
\[
  \alpha_q = \sum_{j = 0}^n (\alpha_{x_j}^q x_j + \alpha_{y_j}^q y_j)
\qquad\text{and}\qquad
  \beta_q = \sum_{j = 0}^n (\beta_{x_j}^q x_j + \beta_{y_j}^q y_j)
\]
with $\alpha_{x_j}^q, \alpha_{y_j}^q, \beta_{x_j}^q, \beta_{y_j}^q \in \bbC$, for $j = 0, \ldots, n$ and $q = 1, \ldots, n+k$. We first consider the 
$\alpha_{x_j}^q$. Comparing coefficients of $x_j^2$ in \eqref{eq:Av=0} shows
\begin{equation} 
\label{eq:alpha=0}
  \alpha_{x_j}^{j+1} = \alpha_{x_j}^{j+2} = \ldots = \alpha_{x_j}^{j+k} = 0, \qquad
j = 0, \ldots, n.
\end{equation}
Comparing coefficients of $x_a \cdot x_b$ in \eqref{eq:Av=0} shows
\begin{equation} 
\label{eq:alpha+alpha=0}
  \alpha_{x_b}^{a+1} + \alpha_{x_a}^{b+1} = \alpha_{x_b}^{a+2} + \alpha_{x_a}^{b+2}
    = \ldots = \alpha_{x_b}^{a+k} + \alpha_{x_a}^{b+k} = 0, \qquad 0 \leq a < b \leq n.
\end{equation}
We know from \eqref{eq:alpha=0} that \eqref{eq:alpha+alpha=0} also holds for $a = b$. This generalization of \eqref{eq:alpha+alpha=0} implies
\[ 
\begin{array}{r@{\qquad}ccl}
  \alpha_{x_j}^q = - \alpha_{x_{q-1}}^{j+1} = \alpha_{x_{j-1}}^{q+1}, & 
    j = 0, \ldots, n, & q = 1, \ldots, n+k, & q - j \leq 0,\\[1ex]
  \alpha_{x_j}^q = - \alpha_{x_{q-k}}^{j+k} = \alpha_{x_{j+1}}^{q-1}, & 
    j = 0, \ldots, n, & q = 1, \ldots, n+k, & q - j \geq k+1.
\end{array} 
\]
Starting from \eqref{eq:alpha=0}, we thus obtain $\alpha_{x_j}^q = 0$ for $q-j \geq 1$, by induction on $q-j$, and also $\alpha_{x_j}^q = 0$ for 
$q-j \leq 0$, by descending induction on $q-j$. This proves
\begin{equation} 
\label{eq:alpha_x=0}
  \alpha_{x_j}^q = 0, \qquad j = 0, \ldots, n, \quad q = 1, \ldots, n+k.
\end{equation}
We next consider the $\beta_{y_j}^q$. Comparing coefficients of all products $y_j^2$ and $y_a \cdot y_b$ in the equations \eqref{eq:Av=0} yields a system of 
linear equations
\begin{equation} 
\label{eq:Mbeta=0}
  M_{\eps} \cdot (\beta_{y_j}^q) = 0
\end{equation}
where $(\beta_{y_j}^q)$ is the column vector consisting of the $(n+1)(n+k)$ unknowns $\beta_{y_j}^q$, and the entries of the coefficient matrix $M_{\eps}$ 
depend continuously on $\eps$.
\par 
For $\eps = 0$, the equations \eqref{eq:Av=0} are symmetric in the $\alpha_q$ and $\beta_q$. Therefore, $M_0$ is also the coefficient matrix of the system of 
linear equations \eqref{eq:alpha=0} and \eqref{eq:alpha+alpha=0}. The latter system has only the trivial solution \eqref{eq:alpha_x=0}, so $M_0$ has maximal 
rank.
\par 
Hence, $M_{\eps}$ has maximal rank for all sufficiently small $\eps \in \bbC$, and, therefore, the system of linear equations \eqref{eq:Mbeta=0} has only the 
trivial solution
\begin{equation} 
\label{eq:beta_y=0}
  \beta_{y_j}^q = 0, \qquad j = 0, \ldots, n, \quad q = 1, \ldots, n+k.
\end{equation}
We finally consider the $\alpha_{y_j}^q$ and the $\beta_{x_j}^q$. Due to \eqref{eq:alpha_x=0} and \eqref{eq:beta_y=0}, we have
\begin{equation*}
  \alpha_q = \sum_{j = 0}^n \alpha_{y_j}^q y_j \qquad\text{and}\qquad
  \beta_q = \sum_{j = 0}^n \beta_{x_j}^q x_j, \qquad q = 1, \ldots, n+k.
\end{equation*}
Comparing coefficients of $x_a \cdot y_b$ in the $i$-th equation of the system \eqref{eq:Av=0} shows
\begin{align} 
\label{eq:beta=-alpha} 
\begin{split}
  \alpha_{y_b}^{i+a} + \beta_{x_a}^{k+1-i+b} = 0, \qquad & a =0, \ldots, n, \quad b
= 0, \ldots, n,\\
    & i = 1, \ldots, k, \quad (1, 1) \neq (b, i) \neq (1, k).
\end{split} 
\end{align}
Also, comparing coefficients of $x_{a+1} \cdot y_b$ in the $(i-1)$-st equation, we deduce
\begin{align*} 
\begin{split}
  \beta_{x_a}^{k+1-i+b} = \beta_{x_{a+1}}^{k+2-i+b}, \qquad & a =0, \ldots, n-1,
\quad b = 0, \ldots, n,\\
    & i = 2, \ldots, k, \quad (1, 2) \neq (b, i) \neq (1, k).
\end{split} 
\end{align*}
Since $n \geq 2$ and $k \geq 3$, all integers $1, \ldots, n+k-1$ occur as $k+1-i+b$ here. Hence, $\beta_{x_j}^q$ depends only on $q-j$, or, in other words
\begin{equation} 
\label{eq:beta=gamma}
  \beta_{x_j}^q = \gamma_{q-j}, \qquad j = 0, \ldots, n, \quad q = 1, \ldots, n+k,
\end{equation}
for some numbers $\gamma_{1-n}, \ldots, \gamma_{n+k} \in \bbC$. Substituting this back into \eqref{eq:beta=-alpha} gives
\begin{align} 
\label{eq:alpha=gamma} 
\begin{split}
  \alpha_{y_j}^q = -\gamma_{k+1+j-q}, \qquad & j = 0, \ldots, n, \quad q = 1,
\ldots, n+k,\\
    & (1, 1) \neq (j, q) \neq (1, n+k).
\end{split} 
\end{align}
Comparing coefficients of $x_a \cdot y_1$ in the first and last equation in \eqref{eq:Av=0} shows
\begin{align}
  \label{eq:gamma_{1-a}}
  0 = \alpha_{y_1}^{1+a} + \beta_{x_a}^{k+1} + \eps \beta_{x_a}^1
    = \alpha_{y_1}^{1+a} + \gamma_{k+1-a} + \eps \gamma_{1-a}, \qquad a = 0, \ldots,
n,\\
  \label{eq:gamma_{n+k-a}}
  0 = \alpha_{y_1}^{k+a} + \beta_{x_a}^2 + \eps \beta_{x_a}^{n+k}
    = \alpha_{y_1}^{k+a} + \gamma_{2-a} + \eps \gamma_{n+k-a}, \qquad a = 0, \ldots, n.
\end{align}
If $a > 0$, then $\alpha_{y_1}^{1+a} = -\gamma_{k+1-a}$ according to \eqref{eq:alpha=gamma}, and hence $\eps \gamma_{1-a} = 0$, by \eqref{eq:gamma_{1-a}}. 
Using the assumption $\eps \neq 0$, we thus obtain
\[
  \gamma_{1-n} = \gamma_{2-n} = \ldots = \gamma_0 = 0.
\]
If $a < n$, then $\alpha_{y_1}^{k+a} = -\gamma_{2-a}$ according to \eqref{eq:alpha=gamma}, and hence $\eps \gamma_{n+k-a} = 0$ by \eqref{eq:gamma_{n+k-a}}. 
Using the assumption $\eps \neq 0$, we thus obtain
\[
  \gamma_{k+1} = \gamma_{k+2} = \ldots = \gamma_{k+n} = 0.
\]
The remaining special cases $a = 0$ of \eqref{eq:gamma_{1-a}}, and $a = n$ of \eqref{eq:gamma_{n+k-a}}, give
\begin{align*}
  \alpha_{y_1}^1 & = -\gamma_{k+1} - \eps \gamma_1 = - \eps \gamma_1,\\
  \alpha_{y_1}^{n+k} & = -\gamma_{2-n} - \eps \gamma_k = - \eps \gamma_k.
\end{align*}
Combined with the last four displayed equations, \eqref{eq:beta=gamma} and \eqref{eq:alpha=gamma} precisely say
\[
  v = \gamma_1 v_1 + \gamma_2 v_2 + \cdots + \gamma_k v_k.
\]
This shows that the vectors $v_1, \ldots, v_k$ span the space ${\rm Syz}_1(A_{\eps})$, as claimed.
\end{proof}
\end{example}
\begin{rem}
The fact that  there are symplectic 't Hooft bundles which are not symplectic RS-instanton bundles follows from Sections 5 and 6 where we prove that the 
family of  symplectic  't Hooft bundles on $\bbP$ with charge $k$ is irreducible of dimension $5kn+4n^2$ and the family of symplectic RS-instanton bundles on 
$\bbP$ with charge $k$ is irreducible of dimension $(4n+2)\cdot k+4n^2+2n-4$.
\end{rem}
\section{The moduli space of  't Hooft instanton bundles and its birational type}
There are a natural parameter space for the data defining a monad for an 't Hooft instanton bundle, an action of a reductive algebraic group on that parameter space,
and a dense open subset of the parameter space which is invariant under the group action, such that two points in this dense open subset lie in the same orbit if and 
only if they define isomorphic 't Hooft instanton bundles (Lemma \ref{lem:GenEquiv}). Thus, the stack and the GIT quotient of the parameter space by the group action 
can be considered as natural moduli spaces for 't Hooft instanton bundles and shall be investigated in this section.
\par 
Consider the symplectic vector space
\[
  U := \bbC^2
\]
and the vector space
\[
  V := \bbC^k.
\]
We identify the matrix $a$ in \eqref{eq:a} with an element
\[
  a \in V^{\oplus(n+k)}.
\]
For each $j$, we identify the pair $l_j, l_j' \in H^0(\bbP, \calO_{\bbP}( 1))$ in \eqref{eq:D} and \eqref{eq:D'} with a linear map
\[
  U^* \longto H^0\bigl(\bbP, \calO_{\bbP}(1)\bigr), \quad j \in \{\,1, \ldots, n+k\,\}.
\]
Choosing a basis of $H^0(\bbP, \calO_{\bbP}( 1))$ identifies these linear maps with elements $L_j \in U^{\oplus (2n+2)}$, $j=1,...,n+k$. We put 
$L = (L_j)_{j=1,...,n+k} \in (U^{\oplus (2n+2)})^{\oplus (n+k)}$. For each element in the vector space
\[
  (U^{\oplus (2n+2)} \oplus V)^{\oplus (n+k)},
\]
we thus obtain a linear map
\[
\begin{CD}
  A\colon  U^* \otimes \calO_{\bbP}^{\oplus (n+k)} @>\diag( L)>> \calO_{\bbP}( 1)^{\oplus (n+k)} @> a>> V \otimes \calO_{\bbP}( 1)
\end{CD}
\]
as in \eqref{eq:A2}. It satisfies $A J A^t = 0$ for the standard symplectic form $J$ on the vector space $(U^*)^{\oplus (n+k)}$.
\par 
Consider the linear algebraic group
\[
  G_{n, k} := (\SL( U) \times \Gm) \wr S_{n+k} \times \GL( V).
\]
Recall that the wreath product $(\SL( U) \times \Gm) \wr S_{n+k}$ is the semidirect product
\[
\xymatrix{
  1 \ar[r]& (\SL( U) \times \Gm)^{\times (n+k)} \ar[r]& {(\SL( U) \times \Gm) \wr S_{n+k}}\ \hskip -1.1cm &
\ar[r] &\ar@(l,r)@/_10pt/[l]&\ \hskip -1.1cm S_{n+k} \ar[r] & 1
}
\]
where $S_{n+k}$ acts on $(\SL( U) \times \Gm)^{\times (n+k)}$ by permuting the factors.
\par 
We can write each element $g \in G_{n, k}$ in the form
\[
  g = (\sigma \cdot (\beta_j, \gamma_j)_j, \alpha)
\]
with $\alpha \in \GL( V)$, $\beta_j \in \SL( U)$, $\gamma_j \in \Gm$ and $\sigma \in S_{n+k}$. It induces an isomorphism
\[
\begin{CD}
  U^* \otimes \calO_{\bbP}^{\oplus (n+k)} @>\diag(g \cdot L) >>
\calO_{\bbP}( 1)^{\oplus (n+k)} @> g \cdot a >>
 V \otimes \calO_{\bbP}( 1)
\\
@V \sigma \circ (\beta_j^*)_j VV @VV  \sigma \circ (\gamma_j)_j V @VV \alpha V
\\
U^* \otimes \calO_{\bbP}^{\oplus (n+k)} @> \diag( L) >> \calO_{\bbP}( 1)^{\oplus (n+k)}
@> a >> V \otimes \calO_{\bbP}( 1)
\end{CD}
\]
with
\begin{align}
  \label{eq:action1} (g \cdot a)_j & := \alpha^{-1}( \gamma_j a_{\sigma( j)}),\\
  \label{eq:action2} (g \cdot L)_j & := \gamma_j^{-1} \beta_j( L_{\sigma( j)}).
\end{align}
These formulas define a linear action of $G_{n, k}$ on $(U^{\oplus(2n+2)} \oplus V)^{\oplus (n+k)}$. The subgroup
\[
  \mu_2 \hookrightarrow G_{n, k}
\]
given by the diagonal embedding $\mu_2 \hookrightarrow (\SL( U) \times \Gm)^{\times(n+k)} \times \GL( V)$ acts trivially.
\par 
Proposition \ref{prop:matrices}, iii), states $R_{n, k} \neq \varnothing$, for the open locus
\[
 R_{n, k} \subseteq (U^{\oplus(2n+2)} \oplus V)^{\oplus (n+k)}
\]
where $A$ has rank $k$ at every point in $\bbP$. The subset $R_{n, k}$ is preserved by $G_{n, k}$. Each point in $R_{n, k}$ defines a symplectic monad \eqref{eq:monad} 
whose cohomology is a symplectic $k$-instanton bundle $E$; it is constant up to isomorphy on the $G_{n,
  k}$-orbits. Note that $E$ has trivial splitting type and is stable, by \cite{ottaviani}, 
Corollary 3.5 and Theorem 3.6.
\begin{defn}
A {\it stable 't Hooft datum} is a point of $R_{n, k}$. 
\end{defn}
We leave it to the reader to formulate the moduli problem of stable 't Hooft data; compare for example \cite{Hof}, Remark 3.5. In the following, we will demonstrate 
that the geometric quotient
\[
R_{n,k}/G_{n, k}
\]
exists as a smooth quasi-projective variety.  For this, we recall some details of the construction of the moduli space ${\rm MI}_{\mathbb{P}^{2n+1}}(k)$ of stable 
symplectic $k$-instanton bundles on $\mathbb{P}^{2n+1}$ contained in \cite{OS}.
\begin{thm}[(The first fundamental theorem of invariant theory for symplectic groups)]
\label{thm:FFT}
Let $E$ and $F$ be finite dimensional complex vector spaces, $\phi\colon F\lra F^*$ a symplectic form on $F$, and $S$ the isometry
group of $(F,\phi)$. Then, the $S$-invariant map
\begin{eqnarray*}
\kappa\colon {\rm Hom}(E,F) &\lra& {\rm Hom}_{\rm AS}(E,E^*)
\\
f &\lma& f^*\circ \phi\circ f
\end{eqnarray*}
induces a closed embedding of the categorical quotient of the vector space ${\rm Hom}(E,\allowbreak F)$ by the action of $S$ into ${\rm Hom}_{\rm AS}(E,E^*)$, the 
sub vector space of antisymmetric homomorphisms. The image consists of those homomorphisms whose rank is less than or equal to ${\rm min}\{\, \dim(E),\dim(F)\,\}$.
\end{thm}
\begin{proof}
\cite{GoWa}, Theorem 5.2.2 and Lemma 5.2.4.
\end{proof}
In the study of monads \eqref{eq:monad}, we look at matrices in the vector space
\[
{\rm M}:={\rm Hom}(U^{\oplus (n+k)},V^{\oplus (2n+2)})
\]
and, in the study of symplectic monads, at matrices in the closed subvariety
\begin{equation}
\label{eq:SympInst}
{\rm SM}:=\bigl\{\, A\in {\rm M}\,|\, AJA^t=0\,\bigr\}.
\end{equation}
According to Theorem \ref{thm:FFT},
\begin{eqnarray*}
 \kappa\colon \bbP({\rm M}) &\lra& \bbP\Bigl({\rm Hom}_{\rm AS}\bigl((V^{\oplus (2n+2)})^*, V^{\oplus(2n+2)}\bigr)\Bigr)
\\
{[A]} &\lma& \bigl[K(A)\colon (V^{\oplus (2n+2)})^*\stackrel{A^*}{\lra} (U^{\oplus (n+k)})^* \stackrel{J}{\lra} U^{\oplus (n+k)}\stackrel{A}{\lra} V^{\oplus (2n+2)}\bigr]
\end{eqnarray*}
is a model for the categorical quotient of $\bbP({\rm M})$ by the given ${\rm Sp}_{2(n+k)}(\CC)$-action, possibly followed by a closed embedding.
\begin{rem}
One calls $K(A)$ the \it Kronecker module \rm of $A$. We may view $K(A)$ as an element of
\[
{\rm H}:= {\rm Hom}\left(\bigwedge\limits^2 (\CC^{2n+2}), {\rm Hom}(V^*, V)\right).
\]
If $A\in {\rm SM}$, then $K(A)$ takes values in the sub vector space ${\rm Hom}_{\rm S}(V^*, V)$ of symmetric homomorphisms, i.e., $K(A)$ is a symmetric Kronecker module 
(compare \cite{OS}, p.\ 39).
\end{rem}
Finally, one has to study the $\SL(V)$-action on $\mathbb{P}({\rm H})$. By work of Hulek \cite{Hul}, a point in $\mathbb{P}({\rm H})$ which comes from a monad whose 
cohomology is a stable instanton bundle is $\SL(V)$-stable (\cite{OS}, Lemmas 1.11 and 1.12). Let
\[
{\rm SI}_k\subset \mathbb{P}({\rm H})
\]
be the $\SL(V)$-invariant locally closed subset of points coming from monads with stable $k$-instanton bundles as cohomology. By the aforementioned result of Hulek, the
geometric quotient
\[
 {\rm MI}_{\mathbb{P}^{2n+1}}(k):={\rm SI}_k/\SL(V)
\]
exists as quasi-projective variety. It is the moduli space of stable $k$-instanton bundles on $\mathbb{P}^{2n+1}$.
\begin{rem}
\label{rem:QuotPriBund}
Set
\[
 \widetilde{\rm SI}_{k}:=\mathbb{P}({\rm SM})\cap \kappa^{-1}({\rm SI}_k)
\]
and let ${\mathcal M}_{\mathbb{P}^{2n+1}}(k)$ be the moduli space of stable vector bundles with Chern character $1/(1-t^2)^k$ on $\mathbb{P}=\mathbb{P}^{2n+1}$.
We have the induced morphism
\[
 \widetilde{\rm SI}_{k}\lra {\mathcal M}_{\mathbb{P}^{2n+1}}(k).
\]
Since this morphism factorizes over the categorical quotient of $\widetilde{\rm SI}_{k}$ by the action of ${\rm Sp}_{2(n+k)}(\mathbb{C})\times {\rm PGL}(V)$,
Proposition \ref{prop:Monadology}, ii), shows that this quotient is actually a geometric one. Remark \ref{rem:InstSimple} implies further that the quotient map 
(compare \cite{OST}, Proposition 1.3.1, or \cite{SchBook}, Theorem 1.5.3.1, i)
\[
  \widetilde{\rm SI}_k\lra  \widetilde{\rm SI}_k/\bigl({\rm Sp}_{2(n+k)}(\CC)\times {\rm PGL}(V)\bigr)=
\widetilde{\rm SI}_k/\bigl({\rm Sp}_{2(n+k)}(\CC)\times \SL(V)\bigr)=
{\rm SI}_k/\SL(V)
\]
is a principal $({\rm Sp}_{2(n+k)}(\CC)\times {\rm PGL}(V))$-bundle.
\end{rem}
We now look at the linear $G_{n,k}$-action on
\[
 W_{n,k}:=(U^{\oplus(2n+2)} \oplus V)^{\oplus (n+k)}.
\]
The subgroup
\[
 \mathbb{G}_m^{\times (n+k+1)}\cong \mathbb{G}_m^{n+k}\times {\mathcal Z}\bigl(\GL(V)\bigr)\subset G_{n,k}
\]
is normal, and we set
\[
 H_{n,k}:=G_{n,k}/ \mathbb{G}_m^{\times (n+k+1)}.
\]
Let us first look at the $\mathbb{G}_m^{\times (n+k+1)}$-action on $W_{n,k}$. We use the isogeny
\begin{eqnarray*}
 \mathbb{G}_m^{\times (n+k)}\times \mathbb{G}_m &\lra&  \mathbb{G}_m^{\times (n+k)}\times \mathbb{G}_m
\\
\bigl((\gamma_j)_{j=1,...,n+k},z\bigr) &\lma& \bigl((z^{-1}\cdot \gamma_j)_{j=1,...,n+k},z^{2}\bigr).
\end{eqnarray*}
Then, the last factor just acts by multiplying everything by $z$, $z\in \mathbb{G}_m$, and the $\mathbb{G}_m^{\times (n+k+1)}$-semistable points in 
$W_{n,k}$ correspond to the $\mathbb{G}_m^{\times (n+k)}$-semistable points in $\mathbb{P}(W_{n,k})$.
\begin{lemma}
\label{lem:SemStab1}
For a point $p=[(L_j)_{j=1,...,n+k}, a]\in \mathbb{P}(W_{n,k})$, the following conditions are equivalent:
\par
{\rm i)} The point $p$ is semistable with respect to the action of $\mathbb{G}_m^{\times (n+k)}$.
\par
{\rm ii)} The point $p$ is stable with respect to the action of $\mathbb{G}_m^{\times (n+k)}$.
\par
{\rm iii)} $\forall j\in\{\, 1,...,n+k\,\}:$ $a_j\neq 0$ and $L_j\neq 0$.
\end{lemma}
\begin{proof}
By \eqref{eq:action1} and \eqref{eq:action2}, the $j$-th copy of $\mathbb{G}_m$ in $\mathbb{G}_m^{\times (n+k)}$ multiplies the $j$-th column of $a\in V^{\oplus (n+k)}$ 
by $z$ and $L_j$ by $z^{-1}$, $z\in \mathbb{G}_m$, $j=1,...,n+k$. With this information about the group action, the proof is straightforward.
\end{proof}
As before, we may view a tuple $(L_j)_{j=1,...,n+k}\in (U^{\oplus (2n+2)})^{\oplus (n+k)}$ as a matrix $(D|D')$ where $D$ and $D'$ are diagonal matrices of the format 
$(n+k)\times (n+k)$ with entries in $H^0(\bbP, \calO_\bbP(1))$. We define
\begin{eqnarray*}
 \alpha\colon W_{n,k} &\lra& {\rm M}
\\
\bigl((L_j)_{j=1,...,n+k}, a\bigr) &\lma& a\cdot (D|D').
\end{eqnarray*}
The map $\alpha$ is invariant under the $\mathbb{G}_m^{\times (n+k)}$-action, and the semistability condition in Lemma \ref{lem:SemStab1} clearly implies that
$\alpha(p)\neq 0$ holds for every $\mathbb{G}_m^{\times (n+k)}$-semistable point $p=((L_j)_{j=1,...,n+k}, a)\in W_{n,k}$. Therefore, $\alpha$ descends to a morphism
\[
 \ol\alpha\colon \ol W_{n,k}:= W_{n,k}/\mathbb{G}_m^{\times (n+k+1)} \lra \bbP({\rm M})
\]
between projective varieties, and the image of $\ol\alpha$ is contained in the closed subvariety $\bbP({\rm SM})$. Using Lemma \ref{lem:SemStab1}, one checks that 
$\ol\alpha$ is injective, too.
\par
The group $H_{n,k}$ acts on the quotient $\ol W_{n,k}$. There is a natural linearization of this action. Note that the vector space $U$ is equipped with the symplectic 
form $U\times U\lra\C$, $(u,u')\lma \det(u|u')$. Our convention is also such that the symplectic form on $U^{\oplus (n+k)}$ is the direct sum of the symplectic forms on 
the summands. In this way, $H_{n,k}$ becomes a subgroup of ${\rm Sp}_{2(n+k)}(\CC)\times {\rm PGL}(V)$.
\begin{rem}
It is straightforward to compute the GIT-semistable and stable points for the $H_{n,k}$-action on $\ol W_{n,k}$. Unfortunately, the condition one finds is weaker than the
condition of $({\rm Sp}_{2(n+k)}(\CC)\times {\rm PGL}(V))$-semistability and stability on $\mathbb{P}({\rm M})$. For example, using the first fundamental theorem of 
invariant theory (Theorem \ref{thm:FFT}), we can determine the $\SL(U)^{\times (n+k)}$- and ${\rm Sp}_{2(n+k)}(\CC)$-semistable points. The 
${\rm Sp}_{2(n+k)}(\CC)$-semistable points in $\bbP({\rm M})$ are just those points
\[
[A]\in \mathbb{P}({\rm M}),\q A\colon U^{\oplus (n+k)}\lra V^{\oplus (2n+2)},
\]
for which the image of $A^*$ is not an isotropic subspace of $( U^{\oplus (n+k)})^*$. Viewing $A$ as a tuple of linear maps
\[
A_i\colon U\lra V^{\oplus (2n+2)},\q i=1,...,n+k,
\]
we see that $[A]$ is $\SL(U)^{\times (n+k)}$-semistable, if and only if the image of $A_i^*$ is not an isotropic subspace of $U$, i.e., $A_i^*$ is surjective, 
$i=1,...,n+k$. It is now clear that the $\SL(U)^{\times (n+k)}$-semistability of $[A]$ will in general not imply its ${\rm Sp}_{2(n+k)}(\CC)$-semistability, not even if 
$[A]$ lies in the image of $\ol\alpha$.
\par
Unfortunately, this means that the GIT-quotient $\ol W_{n,k}\catqot H_{n,k}$ will not map to the GIT-quotient 
$\mathbb{P}({\rm M})\catqot ({\rm Sp}_{2(n+k)}(\CC)\times {\rm PGL}(V))$. This makes it harder to understand the map from the moduli space of stable 't Hooft data to the 
moduli space of stable symplectic instanton bundles (compare Proposition \ref{prop:ModHooft}).
\par 
One should compare this situation with the situation of instantons in general. On the one hand, monads can be considered as quiver representations, and King's \cite{Ki} 
results give a parameter dependent notion of stability for these monads. The results in \cite{CosOtt} show that the monad of an instanton bundle is stable with respect to 
all parameters. On the other hand, it is still unknown, if an instanton bundle is stable as a vector bundle.  
\end{rem}
On the positive side, the facts that $\ol\alpha$ is finite and equivariant and that $H_{n,k}$ is a subgroup of ${\rm Sp}_{2(n+k)}(\CC)\times {\rm PGL}(V)$ imply that a 
point $p\in  W_{n,k}/\mathbb{G}_m^{\times (n+k+1)}$ whose image $\ol\alpha(p)\in\mathbb{P}(M)$ is ${\rm Sp}_{2(n+k)}(\CC)\times {\rm PGL}(V)$-stable is $H_{n,k}$-stable.
This applies, in particular, to the points of the open subset
\[
 \ol R_{n,k}:=R_{n,k}/\mathbb{G}_m^{\times (n+k+1)}\subset  W_{n,k}/\mathbb{G}_m^{\times (n+k+1)}=\ol W_{n,k}
\]
(compare Remark \ref{rem:QuotPriBund}). Therefore, the geometric quotient
\[
 \ol R_{n,k}/H_{n,k}
\]
exists as a quasi-projective variety. Remark \ref{rem:QuotPriBund} also implies that the quotient map
\[
 \ol R_{n,k}\lra \ol R_{n,k}/H_{n,k}
\]
is a principal $H_{n,k}$-bundle. Since $R_{n,k}\lra \ol R_{n,k}$ is a $\mathbb{G}_m^{\times (n+k+1)}$-bundle, the variety $\ol R_{n,k}/H_{n,k}$ is actually smooth.
As in \cite{OST}, Proposition 1.3.1, or \cite{SchBook}, Theorem 1.5.3.1, i), one checks
\[
 \ol R_{n,k}/H_{n,k}\cong \mathbb{P}(W_{n,k})/(G_{n,k}/\mathbb{G}_m)\cong W_{n,k}/G_{n,k}.
\]
\par
In the study of the birational geometry of this quotient, we will need the following open subset: Let
\[
  R_{n, k}^0 \subseteq R_{n, k} \subseteq (U^{\oplus(2n+2)} \oplus V)^{\oplus(n+k)}
\]
be the open locus where, moreover, $l_j$ and $l_j'$ are linearly independent for each $j$, the $n+k$ linear subspaces $\{\, l_j = l_j' = 0\,\} \subseteq \bbP$ are distinct, 
and
\[
  \ker A( 1) \subseteq U^* \otimes \calO_{\bbP}( 1)^{\oplus (n+k)}
\]
has exactly $n+k$ global sections. If $k \geq 3$, then $R_{n, k}^0 \neq \varnothing$ due to Proposition \ref{prop:matrices}, iv).
\begin{lemma}
\label{lem:GenEquiv}
Given two points
  \[
    (a, L), (\tilde{a}, \tilde{L}) \in R_{n, k}^0,
  \]
with corresponding symplectic instanton bundles $E, \tilde{E}$, every isomorphism
  \[
    \varphi\colon \tilde{E} \longto E
  \]
is induced by a unique group element $g \in G_{n, k}$ with $g \cdot (a, L) = (\tilde{a}, \tilde{L})$.
\end{lemma}
\begin{proof}
The isomorphism $\varphi$ comes from a unique isomorphism of symplectic monads
\[
\xymatrix{
    U^* \otimes \calO_{\bbP}^{\oplus (n+k)} \ar[r]^{\tilde{A}} \ar[d]_{\beta^*} & V \otimes
\calO_{\bbP}( 1) \ar[d]^{\alpha}\\
    U^* \otimes \calO_{\bbP}^{\oplus (n+k)} \ar[r]^A & V \otimes \calO_{\bbP}( 1)
  }
\]
with $\alpha \in \GL( V)$ and $\beta \in \Sp( U^{\oplus (n+k)})$; here $A := a \circ \diag(L)$ and $\tilde{A} := \tilde{a} \circ \diag( \tilde{L})$.
\par 
As we have seen in the proof of Proposition \ref{prop:matrices}, i) and ii), the $n+k$ global sections of $\ker A(1)$ and of $\ker \tilde{A}(1)$ are given by the two maps
\[
    J \circ \diag( L)^t \text{ and } J \circ \diag( \tilde{L})^t\colon \calO_{\bbP}^{\oplus (n+k)}
\longto U^* \otimes \calO_{\bbP}( 1)^{\oplus (n+k)}.
\]
Since our isomorphism of symplectic monads has to respect these, it has the form
\[ \xymatrix{
    U^* \otimes \calO_{\bbP}^{\oplus (n+k)} \ar[rr]^{\diag( \tilde{L})} \ar[d]_{\beta^*}
      && \calO_{\bbP}( 1)^{\oplus (n+k)} \ar[r]^{\tilde{a}} \ar[d]^{\gamma} & V \otimes
\calO_{\bbP}( 1) \ar[d]^{\alpha}\\
    U^* \otimes \calO_{\bbP}^{\oplus (n+k)} \ar[rr]^{\diag( L)} && \calO_{\bbP}( 1)^{\oplus (n+k)}
\ar[r]^a & V \otimes \calO_{\bbP}( 1)
} 
\]
with $\gamma \in \GL_{n+k}(\C)$. The loci in $\bbP$ where $\diag( L)$ and $\diag(\tilde{L})$ are not surjective are
\[
    \bigcup_j \{\, l_j = l_j' = 0\,\} \quad\text{and}\quad \bigcup_j \{\,\tilde{l}_j =
\tilde{l}_j' = 0\,\}.
\]
Since our isomorphisms $\beta$ and $\gamma$ have to respect these, they are of the form
\[
    \beta = \sigma \circ (\beta_j)_j \quad\text{and}\quad \gamma = \sigma \circ
(\gamma_j)_j
\]
with $\beta_j \in \GL( U)$, $\gamma_j \in \Gm$ and $\sigma \in S_{n+k}$. Here, $\beta_j \in \SL( U)$ as $\beta$ is symplectic.
\end{proof}
We summarize our discussion:
\begin{prop}
 \label{prop:ModHooft}
Fix positive integers $n$ and $k$. Then, the moduli space
\[
 {\rm HI}_{\mathbb{P}^{2n+1}}(k):=R_{n,k}/G_{n, k}
\]
of stable 't Hooft data exists as a smooth quasi-projective scheme and comes with a generically injective morphism $\iota$ to the moduli space 
${\rm MI}_{\mathbb{P}^{2n+1}}(k)$ of stable symplectic $k$-instanton bundles on $\mathbb{P}^{2n+1}$. The moduli space ${\rm HI}_{\mathbb{P}^{2n+1}}(k)$ contains
\[
 {\rm HI}^0_{\mathbb{P}^{2n+1}}(k):= {R^0_{n, k}}/{G_{n, k}}
\]
as a dense open subscheme. The morphism $\iota$ is injective on ${\rm HI}^0_{\mathbb{P}^{2n+1}}(k)$.
\end{prop}
We are now going to investigate the birational geometry of the moduli space ${\rm HI}_{\mathbb{P}^{2n+1}}(k)$. A {\it quotient of a vector space} {\it modulo a group} 
{\it that acts generically freely} is the quotient of an open subvariety where the action is free; this is well-defined up to birational equivalence.
\begin{thm}
\label{thm:birational}
Assume $k \geq 3$. The coarse moduli scheme $R^0_{n, k}/G_{n, k}$ of `generic' 't Hooft data on $\bbP^{2n+1}$ is birational to
  \[
    \bbC^{5kn + 4n^2 - 2^e(2^e+3)/2} \times \left(\frac{\Mat^{\sym}_{2^e \times 2^e}(\bbC)^2}{{\rm PO}_{2^e}}\right)
  \]
where $2^e$ is the largest power of $2$ that divides both $n$ and $k$.
\par 
Here, the projective orthogonal group ${\rm PO}_{2^e} := \Orth_{2^e}/\mu_2$ acts on the vector space $\Mat^{\sym}_{2^e \times 2^e}( \bbC)^2$ of pairs of symmetric
matrices by simultaneous conjugation.
\end{thm}
\begin{cor}
\label{cor:ModRat}
Assume $k \geq 3$ and let  $R^0_{n, k}/G_{n, k}$ be the coarse moduli scheme of `generic' 't Hooft data on $\bbP^{2n+1}$. It holds:
\par
{\rm i)} If ${\rm gcd}(n,k)\not\equiv 0\ {\rm mod}\, 4$, then $R^0_{n, k}/G_{n, k}$ is rational.
\par
{\rm ii)} If ${\rm gcd}(n,k)\not\equiv 0\ {\rm mod}\,16$, then $R^0_{n, k}/G_{n, k}$ is stably rational.
\end{cor}
\begin{proof}
Saltman \cite{saltman} proved that the quotient $\Mat^{\sym}_{2^e \times 2^e}( \bbC)^2/\allowbreak {\rm PO}_{2^e}$ is rational, for $2^e = 1$, $2$, and stably rational, 
for $2^e = 4$, and Beneish \cite{beneish} that it is stably rational for $2^e = 8$. This and Theorem \ref{thm:birational} imply the corollary.
\end{proof}
\begin{proof}[Proof of Theorem {\rm\ref{thm:birational}}]
By Lemma \ref{lem:GenEquiv} and Proposition \ref{prop:ModHooft}, we have
\[
    \frac{R^0_{n, k}}{G_{n, k}} \simeq \frac{(U^{\oplus (2n+2)} \oplus V)^{\oplus (n+k)}}{G_{n, k}/\mu_2}
\]
where $G_{n, k} = (\SL( U) \times \Gm) \wr S_{n+k} \times \GL( V)$ acts via the formulas \eqref{eq:action1} and \eqref{eq:action2}.
\par 
Let $u_1, u_2$ be a basis of $U$, and put $u = (u_1, u_2) \in U^2$. Then, the $G_{n, k}$-orbit of $(u, \ldots, u) \in (U^2)^{n+k}$ is open, with stabilizer
$(\mu_2 \wr S_{n+k}) \times \GL( V) \subseteq G_{n, k}$. Hence,
  \[
    \frac{R^0_{n, k}}{G_{n, k}} \simeq \frac{( U^{2n} \oplus V)^{n+k}}{\bigl((\mu_2 \wr S_{n+k}) \times \GL(V)\bigr)/\mu_2}
  \]
where each copy of $\mu_2$ acts trivially on $U^{2n}$ and via its nontrivial character on $V$.
\par 
Here, the group action is already generically free on the direct summand $V^{n+k}$, because the finite group $(\mu_2 \wr S_{n+k})/\mu_2$ acts effectively,
and, hence, generically freely, on $V^{n+k}/\GL( V) \simeq \Gr_k( \bbC^{n+k})$. Thus, the no-name lemma (\cite{BK}, Lemma 1.2) yields
\begin{align*}
    \frac{R^0_{n, k}}{G_{n, k}} & \simeq \frac{V^{n+k}}{\bigl((\mu_2 \wr S_{n+k}) \times \GL(V)\bigr)/\mu_2} \times \bbC^{4n(n+k)}\\
      & \simeq \frac{(\bbC \oplus V)^{n+k}}{\bigl((\mu_2 \wr S_{n+k}) \times \GL(V)\bigr)/\mu_2} \times \bbC^{(4n-1)(n+k)}
\end{align*}
where each $\mu_2$ acts trivially on the corresponding summand $\bbC$.
\par  
Sending $\lambda_1, \ldots, \lambda_{n+k} \in \bbC$ to the matrix $\diag(\lambda_1, \ldots, \lambda_{n+k})$ defines a morphism
\[
    \xymatrix{*!{\bbC^{n+k}} \ar@<.7ex>[r] & *!{\Mat^{\sym}_{(n+k) \times (n+k)}( \bbC)}}.
\]
This morphism is equivariant for the above (permutation) action of $\mu_2 \wr S_{n+k}$ on $\bbC^{n+k}$ and its action as a subgroup of $\Orth_{n+k}$ on 
$\Mat^{\sym}_{(n+k) \times (n+k)}( \bbC)$ by conjugation.
\par 
Given $v_1, \ldots, v_{n+k} \in V$, they define a linear map $V^* \lra \bbC^{n+k}$; sending them to the image of this map whenever it is injective defines a dominant
rational map
\[
    \xymatrix{*!{V^{n+k}} \ar@<.7ex>@{-->}[r] & *!{\Gr_k( \bbC^{n+k})}}
\]
which is invariant under $\GL( V)$ and equivariant under $\mu_2 \wr S_{n+k}$ (this time acting as a subgroup of $\Orth_{n+k}$, so each $\mu_2$ acts non-trivially
on one copy of $\bbC$). Taking the product of this morphism and this rational map, we thus obtain a rational map
\begin{equation} 
\label{eq:map}
    \xymatrix{\displaystyle \frac{(\bbC \oplus V)^{n+k}}{\bigl((\mu_2 \wr S_{n+k}) \times \GL(V)\bigr)/\mu_2} \ar@{-->}[r]
      & \displaystyle \frac{\Mat^{\sym}_{(n+k) \times (n+k)}( \bbC) \times \Gr_k(
\bbC^{n+k})}{{\rm PO}_{n+k}}}.
\end{equation}
The symmetric matrices in question correspond to endomorphisms of $\bbC^{n+k}$ which are self-adjoint for the standard bilinear form. It follows that eigenvectors 
with different eigenvalues are orthogonal. A generic self-adjoint endomorphism has no multiple eigenvalues and its eigenvectors are not isotropic. Hence, it admits an 
orthonormal basis consisting of eigenvectors, which is unique up to permutation and signs. This means that the rational map \eqref{eq:map} is indeed birational, so
\[
    \frac{R^0_{n, k}}{G_{n, k}} \simeq \frac{\Mat^{\sym}_{(n+k) \times (n+k)}( \bbC)
\times \Gr_k( \bbC^{n+k})}{{\rm PO}_{n+k}} \times \bbC^{(4n-1)(n+k)}.
\]
Now, the theorem is a consequence of the following proposition.
\end{proof}
\begin{prop}
Suppose $d = d_1 + d_2$ with $d_1, d_2 \geq 1$, and $\nu \in \{\,1, 2\,\}$. Then,  
\[
    \frac{\Mat^{\sym}_{d \times d}( \bbC) \times \Gr_{d_{\nu}}( \bbC^d)}{{\rm PO}_d} \simeq
    \frac{\Mat^{\sym}_{2^e \times 2^e}( \bbC)^2}{{\rm PO}_{2^e}} \times \bbC^{d + d_1 d_2
- 2^e(2^e+3)/2}
\]
where $2^e$ is the largest power of $2$ that divides both $d_1$ and $d_2$.
\end{prop}
\begin{proof}
We endow $\bbC^d$ with the standard symmetric bilinear form $b_d\colon \bbC^d \times \bbC^d \lra \bbC$. Fix $\nu = 1$ or $\nu = 2$. Let a general linear subspace 
$U_{\nu} \subseteq \bbC^d$ of dimension $d_{\nu}$ be given. Then, $U_{3 - \nu} := U_{\nu}^{\perp}$ has dimension $d_{3 - \nu}$, and
\[
    \bbC^d = U_1 \oplus U_2.
\]
Let moreover $f\colon \bbC^d \lra \bbC^d$ be a generic self-adjoint endomorphism. With respect to the above orthogonal direct sum decomposition, we write
\[
    f = \left( \begin{smallmatrix} f_{11} & f_{12} \\ f_{21} & f_{22}
\end{smallmatrix} \right)\colon U_1 \oplus U_2 \longto U_1 \oplus U_2.
\]
The self-adjointness $f = f^*$ means $f_{11} = f_{11}^*$, $f_{22} = f_{22}^*$ and $f_{21} = f_{12}^*$. We may assume $d_1 \geq d_2$ without loss of generality. 
Since $f$ is generic, $f_{12}\colon U_2 \lra U_1$ is then injective. We choose an orthonormal basis of $U_1$.
\par 
In the case $d_1 > d_2$, we define a dominant rational map
\begin{equation} \label{eq:rational} \xymatrix{
    \displaystyle \frac{\Mat^{\sym}_{d \times d}( \bbC) \times \Gr_{d_{\nu}}(
\bbC^d)}{{\rm PO}_d} \ar@{-->>}[r]
      & \displaystyle \frac{\Mat^{\sym}_{d_1 \times d_1}( \bbC) \times \Gr_{d_2}(
\bbC^{d_1})}{{\rm PO}_{d_1}},
} 
\end{equation}
by sending the pair $( f, U_{\nu})$ to $f_{11}\colon U_1 \lra U_1$ and $U_2' := f_{12}(U_2) \subseteq U_1$, where we identify $U_1$ with $\bbC^{d_1}$ via the chosen 
basis; this is well-defined modulo ${\rm PO}_{d_1}$.
\par 
Using the isomorphism $f_{12}\colon U_2 \lra U_2'$, the restriction of $b_d$ to $U_2 \times U_2$ and the map $f_{22}\colon U_2 \lra U_2$ yield a non-degenerate symmetric
bilinear form on $U_2'$, and a self-adjoint endomorphism $f_{22}'$ of $U_2'$. From all these, we can reconstruct $f$ up to ${\rm PO}_d$ by choosing an orthogonal
isomorphism $U_1 \oplus U_2' \cong \bbC^d$.
\par 
Given a generic point $( f_{11}\colon U_1 \lra U_1, U_2' \subseteq U_1)$ in the image of our rational map \eqref{eq:rational}, we have just seen that the fiber over 
it parameterizes non-degenerate symmetric bilinear forms on $U_2'$ together with self-adjoint endomorphisms $f_{22}'$ of $U_2'$. Since $\mu_2 \subseteq \Orth_{d_1}$ acts 
by its nontrivial character on $U_2'\subseteq U_1$, it acts trivially on these forms and endomorphisms; hence, the rational map \eqref{eq:rational} is birationally a 
tower of two vector bundles, both of rank $d_2\cdot (d_2 + 1)/2$. This proves
\begin{equation}
  \label{eq:birational_1}
    \frac{\Mat^{\sym}_{d \times d}( \bbC) \times \Gr_{d_{\nu}}( \bbC^d)}{{\rm PO}_d} \simeq
      \frac{\Mat^{\sym}_{d_1 \times d_1}( \bbC) \times \Gr_{d_2}(
  \bbC^{d_1})}{{\rm PO}_{d_1}} \times \bbC^{d_2(d_2+1)},
\end{equation}
for $\nu = 1$ and for $\nu = 2$; recall that we are assuming $d = d_1 + d_2$ with $d_1 > d_2$.
\par 
In the case $d_1 = d_2$, genericity of $f$ implies that $f_{21}\colon U_1 \lra U_2$ is an isomorphism. We define a dominant rational map
\begin{equation}
  \label{eq:rational'}
  \xymatrix{
    \displaystyle \frac{\Mat^{\sym}_{d \times d}( \bbC) \times \Gr_{d_1}(
\bbC^d)}{{\rm PO}_d} \ar@{-->>}[r]
      & \displaystyle \frac{\Mat^{\sym}_{d_1 \times d_1}( \bbC)^2}{{\rm PO}_{d_1}}
}
\end{equation}
by sending $( f, U_1)$ to $f_{11}\colon U_1 \lra U_1$ and $f_{21}^*( b_d)\colon U_1 \times U_1 \lra \bbC$, which correspond to symmetric matrices via the chosen basis of 
$U_1$; this is well-defined modulo ${\rm PO}_{d_1}$.
\par 
The endomorphism $f_{22}' := f_{21}^{-1} \circ f_{22} \circ f_{21}$ of $U_1$ is self-adjoint with respect to $f_{21}^*( b_d)$. From all these, we can again reconstruct 
$f$ up to ${\rm PO}_d$. Since $\mu_2 \subseteq\Orth_{d_1}$ acts trivially on these endomorphisms $f_{22}'$, the rational map \eqref{eq:rational'} is birationally a vector 
bundle of rank $d_1\cdot (d_1 + 1)/2$. This proves
\begin{equation}
\label{eq:birational_2}
    \frac{\Mat^{\sym}_{d \times d}( \bbC) \times \Gr_{d_1}( \bbC^d)}{{\rm PO}_d} \simeq
      \frac{\Mat^{\sym}_{d_1 \times d_1}( \bbC)^2}{{\rm PO}_{d_1}} \times
\bbC^{d_1(d_1+1)/2}
\end{equation}
under the assumption $d = d_1 + d_2$ with $d_1 = d_2$.
\par 
Now, let's return to the general case $d = d_1 + d_2$ with $d_1, d_2 \geq 1$. Following the Euclidean algorithm as it computes $h := \gcd( d_1, d_2)$, and composing with 
the corresponding birational equivalence \eqref{eq:birational_1} or \eqref{eq:birational_2} in each step, we get
\begin{equation}
\label{eq:birational_3}
    \frac{\Mat^{\sym}_{d \times d}( \bbC) \times \Gr_{d_{\nu}}( \bbC^d)}{{\rm PO}_d} \simeq
      \frac{\Mat^{\sym}_{h \times h}( \bbC)^2}{{\rm PO}_h} \times \bbC^{d + d_1 d_2 -
h(h+3)/2}.
\end{equation}
Recall that $2^e$ is the largest power of $2$ dividing $h$. If $h = 2^e$, then we are done, so we assume $2^e < h$. Since the action of ${\rm PO}_h$ is generically free 
here, the stack quotient $[\Mat^{\sym}_{h \times h}( \bbC)^2/\Orth_h]$ is generically a $\mu_2$-gerbe over our birational quotient modulo ${\rm PO}_h$. The standard 
representation $\bbC^h$ of $\Orth_h$ yields a vector bundle of rank $h$ and nontrivial weight on this $\mu_2$-gerbe, whose index (at the generic point) therefore
divides $h$, and hence divides $2^e$. It follows that the Grassmannian bundle with fibers $\Gr_{2^e}( \bbC^h)$ over this stack quotient has a rational generic fiber, so
\[
    \frac{\Mat^{\sym}_{h \times h}( \bbC)^2}{{\rm PO}_h} \times \bbC^{2^e( h - 2^e)}
      \simeq \frac{\Mat^{\sym}_{h \times h}( \bbC)^2 \times \Gr_{2^e}( \bbC^h)}{{\rm PO}_h}.
\]
Applying \eqref{eq:birational_3} once more, with $h$ and $2^e$ instead of $d$ and $d_{\nu}$, completes the proof.
\end{proof}
\begin{rem}
The same argument shows that the stack quotient $[ R^0_{n, k}/G_{n, k}]$ is birational to 
$\bbC^{5kn + 4n^2 - 2^e(2^e+3)/2} \times [\Mat^{\sym}_{2^e \times 2^e}(\bbC)^2/\Orth_{2^e}]$. In particular, there is a Poincar\'{e} family on some dense open part if 
and only if $2^e = 1$, which means that $n$ or $k$ is odd. Otherwise, the obstruction is a Brauer class of order $2$ and index $2^e$; cf.\ \cite{BCD} and \cite{Amitsur}, 
Theorem 3.
\end{rem}
\section{The moduli space of RS-instanton bundles and its birational type}
In this section, we construct the moduli stack of RS-instanton bundles on $\bbP$ as well as the moduli space of stable RS-instanton bundles, and we determine the 
birational type of these objects. As before, the moduli spaces are obtained as quotients of a parameter space by a group action. This time the group that acts is 
non-reductive.
\par
Put $U := \bbC^2$, and let $p, q \geq 1$ be integers. We consider the multiplication map
\[
  \mu\colon S^p U \otimes S^q U \longto S^{p+q} U.
\]
\begin{lemma}
For a linear hyperplane $\wp \subset S^p U$, the following are equivalent:
\par
{\rm i)} There is a line $\ell \subset U$ with $\wp = \ell \cdot S^{p-1} U$.
\par
{\rm ii)} The restriction $\wp \otimes S^q U \lra S^{p+q} U$ of $\mu$ is not surjective.
\end{lemma}
\begin{proof}
The case $p = 1$ is trivial, so we assume $p \geq 2$. It is obvious that i) implies ii).
\par
For the converse, we assume that there is no line $\ell \subset U$ with $\wp = \ell \cdot S^{p-1} U$. We identify $U$ with $H^0(\bbP^1, \calO(1))$; then
  \[
    \wp \subset S^p U = H^0\bigl(\bbP^1, \calO(p)\bigr).
  \]
The assumption on $\wp$ means that the canonical evaluation map
  \[
    \eta\colon \calO_{\bbP^1}^p \cong \wp \otimes_{\bbC} \calO_{\bbP^1} \longto
\calO_{\bbP^1}( p)
  \]
is surjective. The kernel of $\eta$ is a vector bundle of rank $p-1$ over $\bbP^1$, so
\begin{equation} \label{eq:splitting}
    \ker( \eta) \cong \bigoplus_{i=1}^{p-1} \calO_{\bbP^1}( a_i)
\quad\text{with } a_1, \ldots, a_{p-1} \in \bbZ,
\end{equation}
due to Grothendieck's splitting theorem. We have a short exact sequence
  \[
    0 \longto \ker( \eta) \longto \wp \otimes_{\bbC} \calO_{\bbP^1} \longto[ \eta]
\calO_{\bbP^1}( p) \longto 0
  \]
of vector bundles over $\bbP^1$. The associated long exact cohomology sequence reads
  \[
    0 \longto H^0\bigl(\bbP^1, \ker(\eta)\bigr) \longto \wp \longto S^p U \longto H^1\bigl(\bbP^1, \ker( \eta)\bigr) \longto 0.
  \]
Hence, we conclude $H^0(\bbP^1, \ker( \eta)) = 0$ and $H^1(\bbP^1, \ker(\eta)) \cong \bbC$. Comparing this with the decomposition \eqref{eq:splitting}, we see 
$\ker(\eta) \cong \calO_{\bbP^1}( -1)^{p-2} \oplus \calO_{\bbP^1}( -2)$ and, in particular,
  \[
    H^1\bigl(\bbP^1, \ker( \eta) \otimes \calO_{\bbP^1}(q) \bigr) = 0.
  \]
This implies that the following map induced by $\eta$ is surjective:
  \[
    H^0\bigl(\bbP^1, \wp \otimes_{\bbC} \calO_{\bbP^1}(q)\bigr) \longto H^0\bigl(\bbP^1,\calO_{\bbP^1}(p + q)\bigr).
  \]
But this map is the restriction of the multiplication map $\mu$; therefore this restriction is surjective. This shows that ii) is false if i) is false.
\end{proof}
\begin{cor}
\label{cor:GL_2}
Suppose that three linear maps
  \[
    \alpha \in \GL( S^p U), \quad \beta \in \GL( S^q U) \quad\text{and}\quad
    \gamma \in \GL( S^{p+q} U)
  \]
satisfy $\mu \circ (\alpha \otimes \beta) = \gamma \circ \mu$. Then, there is a linear map $g \in \GL( U)$ such that
  \[
    S^p g \in \bbC^* \cdot \alpha, \quad S^q g \in \bbC^* \cdot \beta
    \quad\text{and}\quad S^{p+q} g \in \bbC^* \cdot \gamma.
  \]
This linear map $g$ is unique up to multiplication by $\bbC^*$.
\end{cor}
\begin{proof}
The map
\begin{eqnarray*}
    \bbP(U) &\longto& \bbP (S^p U)^*
\\
     \ell &\lma& \ell \cdot S^{p-1} U
\end{eqnarray*}
is a closed immersion. Its image is stable under the automorphism of $\bbP (S^p U)^*$ induced by $\alpha$, according to the previous lemma. Thus, $\alpha$ induces 
an automorphism of $\bbP(U)$, which we lift to an automorphism $g$ of $U$. Modifying $\alpha$, $\beta$ and $\gamma$ by $S^p g$, $S^q g$ and $S^{p+q} g$, respectively, 
we may assume $g = \id_U$. This means that the automorphism of $\bbP(S^p U)^*$ induced by $\alpha$ restricts to the identity on $\bbP(U) \subseteq \bbP (S^p U)^*$.
\par 
Given a line $\ell$ in $U$, the hyperplane $\wp := \ell \cdot S^{q-1} U$ in $S^q U$ has the property that $(\ell \cdot S^{p-1} U) \otimes S^q U$ and $S^p U \otimes \wp$ 
have the same image under $\mu$. According to Corollary \ref{cor:GL_2}, this characterizes $\wp$ uniquely. Since $\alpha( \ell \cdot S^{p-1} U) = \ell \cdot S^{p-1} U$ 
due to the previous paragraph, we conclude $\beta( \ell \cdot S^{q-1} U) = \ell \cdot S^{q-1} U$ as well.
\par 
The fundamental theorem of algebra states that every line $\ell$ in $S^p U$ is a product of $p$ lines $\ell_1, \ldots, \ell_p$ in $U$. If $\ell$ is generic, then the 
$\ell_i$ are all distinct, and $\ell$ is the intersection of the hyperplanes $\wp_i = \ell_i \cdot S^{p-1} U$ in $S^p U$. Due to the first paragraph, we have 
$\alpha( \wp_i) = \wp_i$ for all $i$, and, hence, $\alpha( \ell) = \ell$ for generic lines $\ell$. It follows that $\alpha$ induces the identity on $\bbP( S^p U)$, and, 
hence, $\alpha \in \bbC^* \cdot \id$.
\par 
Applying the same arguments to $\beta$, we get $\beta \in \bbC^* \cdot \id$ as well. Because $\mu$ is surjective, this implies $\gamma \in \bbC^* \cdot \id$, too.
\end{proof} 
Let
\[
  f\colon (S^n U)^* \longto H^0\bigl(\bbP, \calO_{\bbP}( 1)\bigr)
\]
be a linear embedding. We fix an integer $k \geq 1$ and consider a linear map
\[
  h\colon S^{n+2k-2} U \longto H^0\bigl(\bbP, \calO_{\bbP}(1)\bigr).
\]
We endow the trivial algebraic vector bundle of rank $2n+2k$ over $\bbP$,
\[
  E^0_{n, k} := \bigl((S^{n+k-1} U)^* \oplus S^{n+k-1} U\bigr) \otimes_{\bbC} \calO_{\bbP},
\]
with the standard symplectic form $J$. We also form the algebraic vector bundles
\[
  E^{-1}_{n, k} := ( S^{k-1} U) \otimes_{\bbC} \calO_{\bbP}( -1)
\quad\text{and}\quad E^1_{n, k} := ( S^{k-1} U)^* \otimes_{\bbC} \calO_{\bbP}(
1)
\]
of rank $k$ over $\bbP$. The linear maps $f$ and $h$ define a morphism of vector bundles
\[
  A_{f, h}\colon E^0_{n, k} \longto E^1_{n, k}
\]
whose induced map of global sections is the direct sum of the following two maps:
\begin{gather}
  \label{eq:F}
( S^{n+k-1} U)^* \longto[ \mu^*] ( S^{k-1} U)^* \otimes (S^n U)^*\xrightarrow{\id \otimes f} ( S^{k-1} U)^* \otimes
H^0\bigl(\bbP, \calO_{\bbP}(1)\bigr),
\\
  \label{eq:H}
S^{n+k-1} U \longto[ \mu_*] ( S^{k-1} U)^* \otimes S^{n+2k-2} U \xrightarrow{\id \otimes h} ( S^{k-1} U)^* \otimes
H^0\bigl(\bbP, \calO_{\bbP}(1)\bigr).
\end{gather}
We assume that $A_{f, h}$ is surjective; this is an open condition on $h$. The composition
\begin{equation}
\label{eq:monad21}
\begin{CD}
  E^{-1}_{n, k} @>J \circ A_{f, h}^*>> E^0_{n, k} @> A_{f, h}>> E^1_{n, k}
\end{CD}
\end{equation}
vanishes; this follows easily from the observation that the bilinear map
\[
  S^{k-1} U \otimes S^{k-1} U \longto[ \mu] (S^n U)^* \otimes S^{n+2k-2} U
\xrightarrow{f \otimes h} H^0\bigl(\bbP, \calO_{\bbP}(2)\bigr)
\]
is symmetric. Thus, for a generic choice of $h$, \eqref{eq:monad21} is a monad and defines a symplectic instanton bundle $E$.
\par 
The special form of the map $A_{f, h}$ means precisely that $E$ is an RS-instanton bundle. To see this, choose a basis of $U$, and endow the symmetric powers of $U$ 
and their duals with the induced bases. Let $f_i, h_i \in H^0( \bbP, \calO_{\bbP}( 1))$ denote the images of these basis vectors under $f, h$. Then, the map 
\eqref{eq:F} corresponds to the special matrix $F$ with entries $f_i$, and the map \eqref{eq:H} corresponds to the persymmetric matrix $H$ with entries
$h_i$. Moreover, for a generic choice of the $h_i$, the $n$-space $\{\,f_0=f_1=\cdots=f_n=0 \,\}$ does not contain any zeros of the maximal minors of $H$, and,
thus, we conclude that $E$ is an RS-instanton bundle.
\par 
Conversely, if $E$ is an RS-instanton bundle, then the entries of the matrices $F$ and $H$ define linear maps $f$ and $h$ as above. This shows that the RS-instanton 
bundles are precisely the instanton bundles arising from monads of the form \eqref{eq:monad21}.
\par 
Let $L_f \subseteq \bbP$ denote the locus where all sections in the image of $f$ vanish. Then, $L_f$ is a linear subspace of dimension $n$ in $\bbP$; it is the linear 
subspace restricted to which $E$ has $n+k$ global sections by construction. If $E$ is generic, then $L_f$ is the only linear subspace of dimension $n$ in $\bbP$ with 
that property; see Proposition \ref{boundsections-sharp}.
\par 
The next question is: When are two such RS-instanton bundles isomorphic? Let
\[
  g \in \GL( U), \quad t \in \bbC^* \quad\text{and}\quad u \in (S^{2n+2k-2} U)^*
\]
be given. They define an isomorphism of symplectic monads
\begin{equation} \label{eq:iso}
\xymatrix{
  E^{-1}_{n, k} \ar[rrrr]^{J \circ A_{f, h}}\ar[d]^-{t^{-1} \cdot S^{k-1} g} & & &&
  E^0_{n, k} \ar[rrrr]^-{A_{f, h}}\ar[d]^-{\left(\begin{smallmatrix} (S^{n+k-1} g^{-1})^* & (\id \otimes u) \circ \mu_*\\0 & S^{n+k-1} g
\end{smallmatrix} \right)} & & & &
  E^1_{n, k} \ar[d]^-{t (S^{k-1} g^{-1})^*}
\\
  E^{-1}_{n, k} \ar[rrrr]_-{J \circ A_{f', h'}^*} &&&& E^0_{n, k} \ar[rrrr]_-{A_{f',h'}} &&&& E^1_{n, k}
}
\end{equation}
where the vertical map in the middle contains as a matrix entry the composition
\[
  S^{n+k-1} U \longto[\mu_*] (S^{n+k-1} U)^* \otimes S^{2n+2k-2} U \xrightarrow{\id
\otimes u} (S^{n+k-1} U)^*
\]
and the monad in the second row comes from the maps
\begin{equation}
\label{eq:action3}
  f' := tf \circ (S^n g)^* \quad\text{and}\quad h' := t ( h  - (u \otimes f)\circ \mu_*) \circ ( S^{n+2k-2} g^{-1}).
\end{equation}
Here, the definition of the map $h'$ involves the composition
\[
  S^{n+2k-2} U \longto[ \mu_*] S^{2n+2k-2} U \otimes (S^n U)^* \xrightarrow{u
\otimes f} H^0\bigl(\bbP, \calO_{\bbP}( 1)\bigr).
\]
A straightforward computation shows that Diagram \eqref{eq:iso} commutes indeed and that the induced isomorphism of RS-instanton bundles is symplectic. Note that 
$L_f = L_{f'}$.
\par 
The formulas \eqref{eq:action3} define a linear group action of the semidirect product
\[
  G := ( \GL( U) \times \bbC^* ) \ltimes (S^{2n+2k-2} U)^*
\]
on the vector space containing the pairs $(f, h)$ considered above, which is
\[
 {\rm RS}:= {\rm RS}_{n,k} := \big( S^n U \oplus (S^{n+2k-2} U)^* \big) \otimes H^0\bigl(\bbP, \calO_{\bbP}(1)\bigr),
\]
such that each group element sending $(f, h)$ to $( f', h')$ induces a symplectic isomorphism \eqref{eq:iso} from the monad given by $A_{f, h}$ to the monad given by 
$A_{f', h'}$. This isomorphism is the identity if and only if the triple $(g, t, u)$ is in the image of the roots of unity $\mu_{n+k-1} \subseteq \bbC^*$ under the 
embedding
\begin{eqnarray*}
  \mu_{n+k-1} &\longto& (\GL( U) \times \bbC^*) \ltimes (S^{2n+2k-2} U)^*
\\
\rho &\lma & ( \rho \cdot \id_U, \rho^{-n}, 0).
\end{eqnarray*}
This image is a normal subgroup, and the resulting quotient group
\[
  G := \frac{(\GL( U) \times \bbC^*) \ltimes (S^{2n+2k-2} U)^*}{\mu_{n+k-1}}
\]
still acts on the vector space ${\rm RS}$.
\begin{prop}
\label{prop:FreeAction}
Given two RS-instanton bundles
  \[
    E = \ker( A_{f, h})/\im( J \circ A_{f, h}^*) \quad\text{and}\quad E' = \ker(
A_{f', h'})/\im( J \circ A_{f', h'}^*)
  \]
with $L_f = L_{f'}$ and a symplectic isomorphism $\varphi\colon E \lra E'$, there is a unique element of $G$ which sends $( f, h)$ to $( f', h')$ and gives back 
$\varphi$ via Diagram \eqref{eq:iso}.
\end{prop}
\begin{proof}
The uniqueness follows from the observation that the isomorphism \eqref{eq:iso} is the identity only if the triple $(g, t, u)$ is in the image of $\mu_{n+k-1}$. 
We prove the existence.
\par 
The given isomorphism $\varphi$ lifts uniquely to an isomorphism
  \[
    \varphi^{\bullet}\colon (E^{\bullet}_{n, k}, A_{f, h}) \longto[ \sim]
(E^{\bullet}_{n, k}, A_{f', h'})
  \]
of symplectic monads, which in turn is given by two linear automorphisms
  \[
    \varphi^0 \in \Sp \big( ( S^{n+k-1} U)^* \oplus S^{n+k-1} U, J \big)
\quad\text{and}\quad \varphi^1 \in \GL \big( ( S^{k-1} U)^* \big)
  \]
satisfying the following equation:
  \begin{equation}
   \label{eq:A}
    A_{f', h'} \circ \varphi^0 = \varphi^1 \circ A_{f, h}.
  \end{equation}
Let's first restrict this equation to $L := L_f = L_{f'} \subset \bbP$. The first summand $( S^{n+k-1} U)^*$ of $H^0( \bbP, E^0_{n, k})$ consists exactly of the 
sections that vanish on $L$. It follows that this summand has to be preserved by $\varphi^0$, so
  \[
    \varphi^0 = \begin{pmatrix} (\gamma^{-1})^* & \tau\\0 & \gamma \end{pmatrix}
  \]
for an appropriate linear automorphism $\gamma \in \GL( S^{n+k-1} U)$ and some symmetric bilinear form $\tau\colon S^{n+k-1} U \lra ( S^{n+k-1} U)^*$.
\par 
The assumption $L_f = L_{f'}$ means that the linear embeddings $f$ and $f'$ have the same image, so there is a linear automorphism $\beta \in \GL( S^n U)$ with
$f' = f \circ \beta^*$. We can write $\varphi^1 = (\alpha^{-1})^*$ for a unique element $\alpha \in \GL( S^{k-1} U)$.
\par 
Now, the first component of Equation \eqref{eq:A} simplifies to $\mu \circ (\alpha \otimes \beta) = \gamma \circ \mu$. This allows us to apply Corollary \ref{cor:GL_2}.
Multiplying the resulting element $g \in \GL( U)$ by a nonzero scalar, if necessary, we can even achieve $\gamma = S^{n+k-1} g$ besides $S^{k-1} g \in \bbC^* \cdot \alpha$ 
and $S^n g \in \bbC^* \cdot \beta$. Comparing the scalars, we get more precisely $\alpha = t^{-1} \cdot S^{k-1} g$ and $\beta = t \cdot S^n g$ for some $t \in \bbC^*$.
\par 
We replace the pair $(f, h)$ by its image under the group element given by the triple $( g, t, 0)$. This reduces us without loss of generality to the case where
  \[
    f = f', \quad \varphi^0 = \begin{pmatrix} \id & \tau\\0 & \id \end{pmatrix}
\quad\text{and}\quad \varphi^1 = \id
  \]
for some symmetric bilinear form $\tau\colon S^{n+k-1} U \lra ( S^{n+k-1} U)^*$. In this situation, the second component of Equation \eqref{eq:A} means that the diagram
\[
\xymatrix{
    S^{k-1} U \otimes S^{n+k-1} U \ar[r]^-{\mu} \ar[d]_{\tau} & S^{n+2k-2} U
\ar[r]^-{h - h'}
      & H^0\bigl( \bbP, \calO_{\bbP}( 1)\bigr) \ar@{=}[d]\\
    S^{k-1} U \otimes (S^{n+k-1} U)^* \ar[r]^-{\mu^*} & (S^n U)^* \ar[r]^-f &
H^0\bigl(\bbP, \calO_{\bbP}( 1)\bigr)
  }
\]
commutes. Since $\mu$ is surjective, we conclude that the image of $h-h'$ is contained in the image of $f$. As $f$ is injective, this implies $h-h' = f \circ \delta$ for
some linear map
  \[
    \delta\colon S^{n+2k-2} U \longto (S^n U)^*.
  \]
Using this, the previous commutative diagram yields the commutative diagram
  \[ \xymatrix{
    S^n U \otimes S^{k-1} U \otimes S^{n+k-1} U \ar[r]^-{\id \otimes \mu}\ar[d]_{\mu \otimes \id} &
    S^n U \otimes S^{n+2k-2} U \ar[d]^{\delta}
\\
    S^{n+k-1} U \otimes (S^{n+k-1} U) \ar[r]^-{\tau} & \bbC.
    }
  \]
The composition $\tau \circ (\mu \otimes \id) = (\id \otimes \mu) \circ \delta$ in this diagram is a multilinear form
  \[
    U^{\otimes (2n+2k-2)} = U^{\otimes n} \otimes U^{\otimes (k-1)} \otimes
U^{\otimes (n+k-1)} \longto \bbC
  \]
which is invariant under the two subgroups $S_{n+k-1} \times S_{n+k-1}$ and $S_n\times S_{n+2k-2}$ in the symmetric group $S_{2n+2k-2}$. But these two subgroups generate 
the full group $S_{2n+2k-2}$, so the multilinear form descends to a linear form
  \[
    u\colon S^{2n+2k-2} U \longto \bbC
  \]
such that $\delta = u \circ \mu$ and $\tau = u \circ \mu$. It follows that the isomorphism $\varphi^{\bullet}$ comes from the element in $G$ given by the triple 
$(\id_U, 1, u)$.
\end{proof}
\begin{cor}
\label{cor:BirRSStack}
The quotient stack $[{\rm RS}/G]$ is birational to the moduli stack of RS-instanton bundles $E$ with charge $k$ over $ \bbP^{2n+1}$.
\end{cor}
\begin{rem}
Note that the quotient stack  $[{\rm RS}/G]$ has dimension
\begin{eqnarray*}
  \dim( {\rm RS}) - \dim( G) & = & \bigl((n+1) + (n+2k-1)\bigr) \cdot (2n+2) 
  - \bigl(4 + 1 + (2n+2k-1)\bigr)
  \\
  & = & (4n+2)\cdot k + 4n^2 + 2n - 4.
\end{eqnarray*}
In the special case $n = 1$, this coincides with the dimension $6k+2$ of the moduli space introduced by Rao and by Skiti in \cite{rao} and \cite{skiti}, respectively.
\end{rem}
Our next aim is to determine the birational type of this moduli stack. For that, we choose a basis of $H^0( \bbP, \calO_{\bbP}( 1))$. This provides a $G$-equivariant 
isomorphism
\[
  {\rm RS} \cong (S^n U)^{\oplus (2n+2)} \oplus ((S^{n+2k-2} U)^*)^{\oplus (2n+2)}.
\]
We start with the quotient of the first summand $(S^n U)^{\oplus (2n+2)}$ modulo the subgroup
\[
  (\GL( U) \times \bbC^*)/\iota( \mu_{n+k-1})
\]
of $G$, whose definition involves the embedding
\begin{eqnarray*}
  \iota\colon \bbC^* &\longto& \GL( U) \times \bbC^*
\\
 \lambda &\lma & (\lambda \cdot \id_U, \lambda^{-n}).
\end{eqnarray*}
Here, the group action is the restriction of the action \eqref{eq:action3}. So the group element represented by $g \in \GL( U)$ and $t \in \bbC^*$ acts on 
$(S^n U)^{\oplus (2n+2)}$ as the automorphism
\begin{eqnarray}
\label{eq:action4}
  (S^n U)^{\oplus (2n+2)} &\longto &(S^n U)^{\oplus (2n+2)}
\\
\nonumber
 f &\lma &(S^n g)( tf).
\end{eqnarray}
\begin{prop}
\label{prop:BirTypeRSStack}
The stack quotient of the vector space $(S^n U)^{\oplus (2n+2)}$ modulo the above linear group action of $(\GL( U) \times \bbC^*)/\iota( \mu_{n+k-1})$ is
\par
{\rm i)} birational to $\bbA^{2n^2 + 4n - 2} \times B\bbC^*$, if $n$ or $k$ is odd;
\par
{\rm ii)} birational to $\bbA^{2n^2 + 4n - 7} \times [\End( U)^2/\GL( U)]$, if $n$ and $k$ are even.
\end{prop}
Here, $B \bbC^*$ is the classifying stack of $\bbC^*$, and $[\End(U)^2/\GL( U)]$ is the stack quotient of $\End( U)^2$ modulo the linear action of $\GL( U)$ by 
simultaneous conjugation.
\begin{proof}
We consider the central extension of linear algebraic groups
  \begin{equation} \label{eq:extension}
    1 \longto \bbC^* \cong \frac{\bbC^*}{\mu_{n+k-1}} \longto[ \iota] \frac{\GL( U)
\times \bbC^*}{\iota( \mu_{n+k-1})}
      \longto[ \pi] \frac{\GL( U)}{\mu_n} \longto 1
  \end{equation}
where $\pi$ sends the class of a pair $(g, t)$ to the class of $t^{1/n} \cdot g \in \GL( U)$; note that the $n$th root $t^{1/n}$ is well-defined modulo $\mu_n$.
Our group action \eqref{eq:action4} is the composition of $\pi$ and the natural action of $\GL( U)/\mu_n$ on $(S^n U)^{\oplus (2n+2)}$ via $n$th symmetric powers.
\par 
Suppose that $n$ is odd. Then, the central extension \eqref{eq:extension} splits, because the map
\begin{eqnarray*}
    \frac{\GL( U) \times \bbC^*}{\iota( \mu_{n+k-1})} &\lra& \frac{\bbC^*}{\mu_{n+k-1}}
\\
      {}[g, t] &\lma&  \left[t \cdot \det(g)^{\frac{n+1}{2}}\right]
\end{eqnarray*}
is a left inverse of $\iota$. Hence, the stack quotient in question is birational to $B\bbC^*$ times the stack quotient of $(S^n U)^{\oplus (2n+2)}$ modulo $\GL(U)/\mu_n$.
Note that already the action of $\GL(U)/\mu_n$ on $(S^n U)^{\oplus 2}$ is generically free. Moreover, the isomorphism
\begin{eqnarray*}
    \GL(U)/\mu_n &\longto[ \sim]&  \GL( U)
\\
{}[g] & \lma & \det(g)^{\frac{n-1}{2}} \cdot g
\end{eqnarray*}
shows that $\GL(U)/\mu_n$ also has a generically free representation of dimension $4$ containing an open orbit. The no-name lemma (\cite{BK}, Lemma 1.2) allows us to 
replace the representation $(S^n U)^{\oplus (2n+2)}$ by one which contains this $4$-dimensional representation as a direct summand, and then to conclude that the quotient 
is rational, as required.
\par 
Now, suppose that $n$ is even and $k$ is odd. Then, \eqref{eq:extension} also splits, since the map
\begin{eqnarray*}
    \frac{\GL( U) \times \bbC^*}{\iota( \mu_{n+k-1})} &\longto & \frac{\bbC^*}{\mu_{n+k-1}}
\\
      {}[g, t] & \lma &\left[ \det(g)^{\frac{n+k-1}{2}}\right]
\end{eqnarray*}
is a left inverse of $\iota$. Hence, the stack quotient in question is again birational to $B\bbC^*$ times the stack quotient of $(S^n U)^{\oplus (2n+2)}$ modulo 
$\GL(U)/\mu_n$. The isomorphism
\begin{eqnarray*}
    \GL(U)/\mu_n &\longto[ \sim]& {\rm PGL}( U) \times \bbC^*
\\
{}[g] &\lma & \left[g, \det( g)^{\frac{n}{2}})\right]
\end{eqnarray*}
provides a generically free action of $\GL(U)/\mu_n$ on $\End( U)^2 \oplus \bbC$, such that the quotient is rational. Using the no-name lemma as before, it follows that 
the quotient of $(S^n U)^{\oplus (2n+2)}$ modulo $\GL(U)/\mu_n$ is also rational, as required.
\par 
Finally, suppose that $n$ and $k$ are both even. We use the group isomorphism
\begin{eqnarray*}
    \frac{\GL( U) \times \bbC^*}{\iota( \mu_{n+k-1})} & \longto[ \sim] & \GL( U) \times\bbC^*
\\
     {}[g, t] & \lma &  \bigl( \det( g)^{\frac{n+k-2}{2}} \cdot g,\det( g)^{\frac{n}{2}} \cdot t\bigr)
\end{eqnarray*}
whose composition with the embedding $\iota$ of $\bbC^* \cong \bbC^*/\mu_{n+k-1}$ is the embedding
\begin{eqnarray*}
    \bbC^* &\longto & \GL( U) \times \bbC^*
\\
\lambda &\lma & ( \lambda \cdot\id_U, 1).
\end{eqnarray*}
Thus, the cokernel of $\iota$ is isomorphic to ${\rm PGL}( U) \times \bbC^*$, which again acts generically freely on $\End( U)^2 \oplus \bbC$. The no-name lemma allows 
us to replace the representation $(S^n U)^{\oplus (2n+2)}$ by the direct sum of $\End( U)^2 \oplus \bbC$ and a trivial representation of dimension
  \[
    (n+1)\cdot (2n+2) - 2 \cdot 4 - 1 = 2n^2 + 4n - 7,
  \]
whose quotient modulo $\GL( U) \times \bbC^*$ is precisely as claimed.
\end{proof}
\begin{cor}
The moduli stack of RS-instanton bundles with charge $k$ over $ \bbP^{2n+1}$ is
\par
{\rm i)} birational to $\bbA^{(4n+2)k + 4n^2 + 2n - 4} \times B\mu_2$, if $n$ or $k$ is odd;
\par
{\rm ii)} birational to $\bbA^{(4n+2)k + 4n^2 + 2n - 9} \times [\End( U)^2/\SL(U)]$, if $n$ and $k$ are even.
\end{cor}
\begin{proof}
The projection onto the first summand
  \[
    {\rm RS} \cong (S^n U)^{\oplus (2n+2)} \oplus \bigl((S^{n+2k-2} U)^*\bigr)^{\oplus (2n+2)} \longto (S^n U)^{\oplus (2n+2)}
  \]
descends to a $1$-morphism of stacks
  \[
    \Phi\colon \left[ \frac{\rm RS}{G} \right] \longto \left[ \frac{(S^n U)^{\oplus (2n+2)}}{(\GL( U) \times \bbC^*)/\iota( \mu_{n+k-1})} \right].
  \]
The latter stack is generically a gerbe with band $\bbC^*$ due to the previous proposition. The $1$-morphism $\Phi$ is a vector bundle $\calV$ with fibers
  \[
    \bigl((S^{n+2k-2} U)^*\bigr)^{\oplus (2n+2)} \big/ f_*\bigl( (S^{2n+2k-2} U)^*\bigr),
  \]
on which the automorphism groups $\bbC^*$ of the gerbe act with weight $2$. Its rank is
  \[
    \rank( \calV) = (n+2k-1)\cdot (2n+2) - (2n+2k-1) = (4n+2)\cdot k + 2n^2 - 2n - 1.
  \]
\par 
Now, suppose that $n$ or $k$ is odd. Let $\calL$ denote the tensor square of the universal line bundle over the classifying stack $B \bbC^*$. Then, the automorphism 
groups $\bbC^*$ of this neutral gerbe $B \bbC^*$ act with the same weight $2$ on the fibers of $\calL$. But any two vector bundles over a $\bbC^*$-gerbe with the same 
rank and the same weight are isomorphic over some dense open substack of that $\bbC^*$-gerbe, according to \cite{par}, Lemma 4.10. Using Proposition 
\ref{prop:BirTypeRSStack}, we conclude that the total space $[{\rm RS}/G]$ of $\calV$ is birational to $\bbA^{2n^2 + 4n - 2}$ times the total space of the vector bundle
  \[
    \calL^{\oplus (4n+2)k + 2n^2 - 2n - 1} \longto B \bbC^*.
  \]
The projection of the line bundle $\calL$ onto $B \bbC^*$ coincides with the natural morphism
  \[
    \phi\colon B \mu_2 \longto B \bbC^*.
  \]
Hence, the total space of $\calL^{\oplus N+1}$ over $B \bbC^*$ is the total space of $\phi^* \calL^{\oplus N}$ over $B \mu_2$. But this pullback vector bundle is trivial, 
because $\mu_2 \subset \bbC^*$ acts trivially on the fibers of $\calL$. This shows that the total space of $\calL^{\oplus N+1}$ over $B \bbC^*$ is isomorphic to 
$\bbA^N \times B \mu_2$. It follows that the stack $[{\rm RS}/G]$ is birational to
  \[
    \bbA^{2n^2 + 4n - 2} \times \bbA^{(4n+2)k + 2n^2 - 2n - 2} \times B \mu_2
  \]
in the case where $n$ or $k$ is odd.
\par 
Now, suppose that $n$ and $k$ are even. Let $\calU$ denote the universal vector bundle of rank $2$ over the quotient stack $[\End( U)^2/\GL( U)]$, or in other words the 
pullback of the universal vector bundle of rank $2$ over the classifying stack $B \GL( U)$. On the line bundle $\det( \calU)$, the automorphism groups $\bbC^*$ of the 
dense open $\bbC^*$-gerbe in $[\End( U)^2/\GL( U)]$ act again with weight $2$. Combining \cite{par}, Lemma 4.10, and Proposition \ref{prop:BirTypeRSStack} as before, 
we conclude that the stack $[{\rm RS}/G]$ is birational to $\bbA^{2n^2 + 4n - 7}$ times the total space of the vector bundle
  \[
    \det( \calU)^{\oplus (4n+2)k + 2n^2 - 2n - 1} \longto[] \big[\End( U)^2/\GL( U)\big].
  \]
The projection of the line bundle $\det( \calU)$ onto $[\End( U)^2/\GL( U)]$ coincides with the natural morphism
  \[
    \psi\colon \big[\End( U)^2/\SL( U)\big] \longto[] \big[\End( U)^2/\GL( U)\big].
  \]
Hence, the total space of $\det( \calU)^{\oplus N+1}$ over $[\End( U)^2/\GL( U)]$ is the total space of $\psi^* \det( \calU)^{\oplus N}$ over $[\End( U)^2/\SL( U)]$.
But this pullback vector bundle is now trivial over a dense open substack of $[\End( U)^2/\SL( U)]$, because the generic stabilizer $\mu_2 \subset \SL( U)$ acts trivially 
on the fibers of $\det( \calU)$; cf.\ \cite{par}, Corollary 4.8. This shows that the total space of $\det( \calU)^{\oplus N+1}$ over $[\End( U)^2/\GL( U)]$ is birational 
to $\bbA^N \times [\End( U)^2/\SL( U)]$. It follows that the stack $[{\rm RS}/G]$ is birational to
  \[
    \bbA^{2n^2 + 4n - 7} \times \bbA^{(4n+2)k + 2n^2 - 2n - 2} \times \big[\End( U)^2/\SL( U)\big]
  \]
in the case where $n$ and $k$ are both even.
\end{proof}
Next, we would like to construct the moduli space for Rao--Skiti instanton bundles as an algebraic variety. We have already set up the group action for this moduli 
problem. Since the group $G$ that acts is not reductive, it seems that the construction of the quotient is not as straightforward as in the setting of 't Hooft instanton 
bundles. However, we can easily reduce to a quotient problem for a reductive group. That problem can be solved with the construction of \cite{OS} and some descent theory. 
So, let us make all this precise. Denote by
\[
 {\rm RS}^{\rm s}_{n,k}\subset {\rm RS}_{n,k}
\]
the $G$-invariant open subset of RS-monads as in \eqref{eq:monad21} whose cohomology is a stable symplectic instanton bundle. By Remark \ref{rem:RSStable}, i), this set 
is non-empty.
\begin{prop}
\label{prop:RSModuli}
For given positive integers $n$ and $k$, the moduli space
\[
 {\rm RSI}_{\bbP^{2n+1}}(k) := {\rm RS}_{n.k}^{\rm s}/G
\]
of stable RS-instanton bundles with charge $k$ on $\bbP^{2n+1}$ exists as a smooth quasi-projective variety. It is equipped with a generically injective morphism
\[
 \iota\colon {\rm RSI}_{\bbP^{2n+1}}(k)\lra {\rm MI}_{\bbP^{2n+1}}(k)
\]
to the moduli space of stable symplectic instanton bundles with charge $k$ on $\bbP^{2n+1}$.
\end{prop}
\begin{proof} 
Formulas \eqref{eq:F} and \eqref{eq:H} describe an embedding
\[
\psi\colon {\rm RS}_{n,k} \lra {\rm SM}
\]
of ${\rm RS}_{n,k}$ into the space of all symplectic monads (compare \eqref{eq:SympInst}). We also have a homomorphism
\[
 \alpha\colon G \lra H,\q H:= {\rm Sp}\bigl((S^{n+k-1}U)^*\oplus S^{n+k-1}U, J\bigr)\times \GL(V).
\]
Note that $\alpha$ is injective (see Proposition \ref{prop:FreeAction} and its proof for the details).
\par
The principal $G$-bundle $H\lra H/G$ and the action of $G$ on ${\rm RS}^{\rm s}_{n,k}$ define the associated fiber bundle
\[
 H\mathop{\times}^G {\rm RS}^{\rm s}_{n,k} \lra H/G
\]
with typical fiber ${\rm RS}^{\rm s}_{n,k}$ (see \cite{Serre}, \S 3.2). It is a smooth quasi-projective variety on which the reductive group $H$ acts from the left. 
By Proposition \ref{boundsections-sharp} and Proposition \ref{prop:FreeAction}, we have a generically injective and $H$-equivariant morphism
\[
 a\colon H\mathop{\times}^G {\rm RS}^{\rm s}_{n,k}\lra {\rm SM}.
\]
Since $H\mathop{\times}^G {\rm RS}^{\rm s}_{n,k}$ is normal, the $H$-action on this space can be linearized in some ample line bundle $A$ (\cite{Mu}, Corollary 1.6).
\par
Let
\[
 \widehat{\rm SI}_k\subset {\rm SM}\setminus \{0\}
\]
be the preimage of $\widehat{\rm SI}_k$ under the projection ${\rm SM}\setminus \{0\}\lra \mathbb{P}({\rm SM})$. Remark \ref{rem:QuotPriBund} shows that
\[
 \widehat{\rm SI}_k\lra {\rm MI}_{\mathbb{P}^{2n+1}}(k)
\]
is a principal $H$-bundle. In particular, it is a faithfully flat morphism. The $H$-action on the space $H\mathop{\times}^G {\rm RS}^{\rm s}_{n,k}$ and the linearization of this 
action in $A$ provide the data which enable us to descend the $H$-equivariant morphism
\[
 H\mathop{\times}^G {\rm RS}^{\rm s}_{n,k}\lra {\rm SM}
\]
to a quasi-projective variety $Q$ and a morphism $Q\lra {\rm MI}_{\mathbb{P}^{2n+1}}(k)$. For this, one applies a descent theorem by Grothendieck 
(\cite{SGA}, Expos\'e VIII, Proposition 7.8; \cite{Luet}, Theorem 7, p.\ 138). Here,
\[
 H\mathop{\times}^G {\rm RS}^{\rm s}_{n,k}\lra Q
\]
is also a principal $H$-bundle. In particular, $Q$ is the categorical quotient for ${\rm RS}^{\rm s}_{n,k}$ with respect to the $G$-action.
\end{proof}
Proposition \ref{prop:BirTypeRSStack} implies the following statements:
\begin{cor}
\label{cor:RSInstRat}
The coarse moduli space ${\rm RSI}_{\bbP^{2n+1}}(k)$ of stable RS-instanton bundles with charge $k$ over $\bbP^{2n+1}$ is rational.
\end{cor}
\begin{cor}
There is a Poincar\'{e} family parameterized by some open subscheme of the coarse moduli space ${\rm RSI}_{\bbP^{2n+1}}(k)$ of
stable RS-instanton bundles with charge $k$ over $\bbP^{2n+1}$ if and only if $n$ is odd or $k$ is odd.
\end{cor}

\end{document}